\documentclass[sn-mathphys,Numbered]{sn-jnl}% Math and Physical Sciences Reference Style
\usepackage{graphicx}%
\usepackage{multirow}%
\usepackage{amsmath,amssymb,amsfonts}%
\usepackage{amsthm}%
\usepackage{mathrsfs}%
\usepackage[title]{appendix}%
\usepackage{xcolor}%
\usepackage{textcomp}%
\usepackage{manyfoot}%
\usepackage{booktabs}%
\usepackage{algorithm}%
\usepackage{algorithmicx}%
\usepackage{algpseudocode}%
\usepackage{listings}%
\theoremstyle{thmstyleone}%
\newtheorem{theorem}{Theorem}%  meant for continuous numbers
\newtheorem{proposition}[theorem]{Proposition}%

\theoremstyle{thmstyletwo}%
\newtheorem{example}{Example}%
\newtheorem{remark}{Remark}%
\newtheorem{corollary}{Corollary}%
\newtheorem{lemma}{Lemma}%

\theoremstyle{thmstylethree}%

\begin{document}

\title[Article Title]{Minimizers of $U$-processes and their domains of attraction}


\author{\fnm{Dietmar} \sur{Ferger}}\email{dietmar.ferger@tu-dresden.de}



\affil{\orgdiv{Fakult\"{a}t Mathematik}, \orgname{Technische Universit\"{a}t Dresden}, \orgaddress{\street{Zellescher Weg 12-14}, \city{Dresden}, \postcode{01069}, \country{Germany}}}

%\affil[2]{\orgdiv{Department}, \orgname{Organization}, \orgaddress{\street{Street}, \city{City}, \postcode{10587}, \state{State}, \country{Country}}}

%\affil[3]{\orgdiv{Department}, \orgname{Organization}, \orgaddress{\street{Street}, \city{City}, \postcode{610101}, \state{State}, \country{Country}}}


\abstract{In this paper, we study the minimizers of U-processes and their domains of attraction. U-processes arise in various statistical contexts, particularly in M-estimation, where estimators are defined as minimizers of certain objective functions.
Our main results establish necessary and sufficient conditions for the distributional convergence of these minimizers, identifying a broad class of normalizing sequences that go beyond the standard square-root asymptotics with normal limits. We show that the limit distribution belongs to exactly one of the four classes
introduced by Smirnov \cite{Smirnov}. These results do not only extend Smirnov's theory but also generalize existing asymptotic theories for M-estimators, including classical results by Huber \cite{Huber1} and extensions to higher-degree U-statistics.
Furthermore, we analyze the domain of attraction for each class, providing alternative characterizations that determine which types of statistical estimators fall into a given asymptotic regime.}




\maketitle

\section{Introduction}\label{sec1}
Let $X_1,\ldots,X_n, n \in \mathbb{N},$ be independent random variables defined on a common probability space $(\Omega,\mathcal{A},\mathbb{P})$
with values in some measurable space $(S,\mathcal{S})$ and joint distribution $P=\mathbb{P}\circ X^{-1}_i$. Consider a
map $h:S^l \times \mathbb{R} \rightarrow \mathbb{R}$ such that $h(\textbf{x},\cdot):\mathbb{R} \rightarrow \mathbb{R}$
is convex for every $\textbf{x}=(x_1,\ldots,x_l) \in S^l$ and $h(\cdot,t):S^l \rightarrow \mathbb{R}$ is $\mathcal{S}^l$-Borel measurable and
symmetric for each $t \in \mathbb{R}$. It induces an \emph{U-process}
\begin{equation} \label{Un}
 U_n^0(t):= \binom{n}{l}^{-1} \sum_{1 \le i_1<\ldots<i_l \le n} h(X_{i_1},\ldots,X_{i_l},t), \; t \in \mathbb{R},
\end{equation}
with pertaining mean-function
\begin{equation} \label{U0}
 U^0(t)=\mathbb{E}[h(X_1,\ldots,X_l,t)]=\int_{S^l}h(\textbf{x},t)d P^l(\textbf{x}), \; t \in \mathbb{R},
\end{equation}
provided the integral in (\ref{U0}) exists, i.e.
\begin{equation} \label{U0exists}
 \int_{S^l}|h(\textbf{x},t)|d P^l(\textbf{x})< \infty.
\end{equation}
In statistics  many parameters of interests are defined or can be represented as
minimizing points of $U^0$ and in this case minimizers of $U_n^0$ are reasonable estimators for these. According to Huber \cite{Huber2}, p.44, it is an advantage to replace $U^0$ with the function
\begin{equation} \label{U}
 U(t):=\int_{S^l}h(\textbf{x},t)-h(\textbf{x},t_0) d P^l(\textbf{x}), \; t \in \mathbb{R},
\end{equation}
where $t_0$ is any fixed real number e.g. $t_0=0$. If $U^0$ exists, then so does $U$ and the respective sets of all minimizing points coincide.
As to the existence of $U$ let $D^+h(\textbf{x},t)$ and $D^-h(\textbf{x},t)$ denote the right-and left derivative, respectively, of $h(\textbf{x},\cdot)$ at point $t$.
We will always assume that
\begin{itemize}
\item[(A0)] $\int_{S^l} |D^\pm h(\textbf{x},t)| P^l(d\textbf{x}) < \infty \quad \forall \; t \in \mathbb{R}.$
\end{itemize}
Conclude from Proposition 2.1 in Ferger \cite{Ferger0} (with $\mathfrak{X}=S^l$ and $Q=P^l$ there) that then $U(t)$ exists and is real-valued for every real $t$ and
that the map $U:\mathbb{R} \rightarrow \mathbb{R}$ is convex. Now the above mentioned advantage is that according to Proposition 2.2 in \cite{Ferger0} condition (A0) in fact is strictly weaker than the integrability condition (\ref{U0exists}).

Moreover, by Proposition 2.4 in \cite{Ferger0}
\begin{equation} \label{Dpm}
 D^\pm U(t)=\int_{S^l} D^\pm h(\textbf{x},t) P^l(d\textbf{x})= \mathbb{E}[D^\pm h(\textbf{X},t)]\quad \forall \; t \in \mathbb{R},
\end{equation}
where $\textbf{X}:=(X_1,\ldots,X_l)$. According to Proposition 2.5 in Ferger \cite{Ferger0} the set Argmin$(U)$ of all minimizing points of $U$ is given by
\begin{align} \label{AU}
\text{Argmin}(U) &=\{m \in \mathbb{R}: U(m)= \inf_{t \in \mathbb{R}} U(t)\} \nonumber\\
 &=\{m \in \mathbb{R}: D^-U(m) \le 0 \le D^+ U(m)\} \\
% &= \{m \in \mathbb{R}:\int_{S^l} D^- h(\textbf{x},m) P^l(d\textbf{x})\le 0 \le \int_{S^l} D^+ h(\textbf{x},m) P^l(d\textbf{x})\}\nonumber\\
 &= \{m \in \mathbb{R}: \mathbb{E}[D^- h(\textbf{X},m)] \le 0 \le \mathbb{E}[D^+ h(\textbf{X},m)]\}. \nonumber
\end{align}

Introduce the empirical measure
$$
 P_n := \binom{n}{l}^{-1} \sum_{1 \le i_1<\ldots<i_l \le n} \delta_{(X_{i_1},\ldots,X_{i_l})},
$$
where $\delta_\textbf{x}$ denotes the Dirac-measure at point $\textbf{x}=(x_1,\ldots,x_l) \in S^l$.
If the integrating measure $P^l$ in the definition of $U(t)$ is replaced by $P_n$, then one obtains
$$
 U_n(t)= \binom{n}{l}^{-1} \sum_{1 \le i_1<\ldots<i_l \le n} h(X_{i_1},\ldots,X_{i_l},t)-h(X_{i_1},\ldots,X_{i_l},t_0), \; t \in \mathbb{R}.
$$
Obviously, $U_n$ differs from $U_n^0$ only up to a constant, whence both functions possess the same minimizing points.
Assume that $h(\textbf{x},t)\rightarrow \infty$ as $|t|\rightarrow \infty$ for all $\textbf{x} \in S^l$. Then $U_n$ is \emph{coercive}, i.e., $U_n(t) \rightarrow \infty$ as $|t| \rightarrow \infty$. In this case the set Argmin$(U_n)$ of all minimizing points of $U_n$ is
a non-empty compact interval. From Proposition 2.5 of Ferger \cite{Ferger0} it follows that
\begin{equation} \label{AUn}
\emptyset \neq \text{Argmin}(U_n)= \{t \in \mathbb{R}: V_n^-(t)\le0\le V_n^+(t)\},
\end{equation}
where
\begin{equation} \label{Vnpm}
 V_n^\pm(t) = \binom{n}{l}^{-1} \sum_{1 \le i_1<\ldots<i_l \le n} D^\pm h(X_{i_1},\ldots,X_{i_l},t).
\end{equation}
In the important special case $l=1$ the above statistics simplify to $V_n^\pm(t)=n^{-1}\sum_{i=1}^n D^\pm h(X_i,t)$.\\

First, we let $m_n$ be the smallest minimizer of $U_n$ (or equivalently of $U_n^0$). Later on we will drop this assumtion. Since by Proposition 2.1 in Ferger \cite{Ferger0} (with $Q=P_n$) the stochastic process $U_n$ is convex, Corollary 3.3 in Ferger \cite{Ferger1} ensures that $m_n$ is a real random variable.
Our main results give several necessary and sufficient conditions for the distributional convergence of $(m_n-m)/a_n$ with some positive sequence $(a_n)$ converging to zero, which can be determined. It is also shown that the distribution function
of the limit variable can belong to exactly one of four possible classes, which are specified precisely.
These classes coincide with those of Smirnov \cite{Smirnov}.
All this we do under the following assumptions:\\

\begin{itemize}
\item[(A1)] $D^+U(m)=0=D^-U(m)$.

If we put $V(t):=D^+U(t)$, then (A1) can be rewritten as $V(m-)=0=V(m)$, because $D^-U(t)=D^+U(t-)=V(t-)$ by Theorem 24.1 in Rockafellar \cite{Rockafellar}. From (A1) one can deduce with (\ref{AU}) that $m$ is a minimizing point of $U$.
\item[(A2)] $\mathbb{E}[D^+h(\textbf{X},m)^2] < \infty$

If (A2) holds, then by Jensen's inequality and Fubini's theorem $$\zeta := \int_S \mathbb{E}[D^+ h(x,X_2,\ldots,X_l,m)]^2 P(dx) \le \mathbb{E}[D^+h(\textbf{X},m)^2] < \infty,\;\; l\ge 2,$$ is a finite non-negative real number.
If $l=1$, then $\zeta:= \mathbb{E}[D^+h(X_1,m)^2] \in [0,\infty)$.
It is supposed that actually

\item[(A3)] $\zeta$ is positive: $\zeta>0$.

Observe that $K(\textbf{x}):= D^+h(\textbf{x},m)$ is symmetric and measurable. Let $K_1(x):=\mathbb{E}[K(x,X_2,\ldots,X_l]), x \in \mathbb{R},$ be
the \emph{first associated function} of the kernel $K$. It follows from (A1) that $\zeta=\text{Var}[K_1(X_1)]$ and so $\sigma^2 := l^2 \zeta$ is the variance of the first term in the Hoeffding-decomposition of the
$U$-statistic with kernel $K$, confer e.g. Koroljuk and Borovskich \cite{Koroljuk} or Serfling \cite{Serfling}.

\item[(A4)] There exist points $t_0 <0 <t_1$ such that $\mathbb{E}[\big(D^+h(\textbf{X},m+t)-D^+h(\textbf{X},m)\big)^2] < \infty$ for $t \in \{t_0,t_1\}$.
\end{itemize}

In the following theorem $(a_n)$ is a sequence of positive real constants with $a_n \rightarrow 0$. Further, given a map $H:\mathbb{R} \rightarrow \mathbb{R}$ the set of all continuity points of $H$ is denoted by $C_H$.
In case that $H$ is an increasing function, then $\tilde{H}$ defined by $\tilde{H}(x):=H(x+)$ is increasing as well, but in addition is right continuous and has left limits (rcll). Moreover, $\tilde{H}=H$ on $C_H$, confer Chung \cite{Chung}, pp.1-5. We say that a distribution function is \emph{degenerated at zero}, if it corresponds to the Dirac measure at zero.
Finally, let $\Phi_\sigma$ denote the distribution function of a centered normal random variable $N(0,\sigma^2)$ with variance $\sigma^2$.\\

\begin{theorem} \label{thm1} Let $H_n$ denote the distribution function of $m_n$. Assume that (A0)-(A4) hold. Recall that $V=D^+U$ and $\sigma^2 = l^2 \zeta$. Then the following statements (1) and (2) are equivalent:
\begin{itemize}
\item[(1)]
\begin{equation} \label{dconv}
       H_n(a_n x+m) \rightarrow H(x) \quad \forall \; x \in C_H,
\end{equation}
where $H$ is some sub distribution function.
\item[(2)]
\begin{equation} \label{nscond}
\delta_n(x):= \sqrt{n} V(m+ a_n x) \rightarrow \delta(x) \in [-\infty,\infty] \quad \forall \; x \in D,
\end{equation}
where $D$ is some dense subset of $\mathbb{R}$.
\end{itemize}
In both cases, $H(x)= \Phi_\sigma(\delta(x))$ for all $x \in \mathbb{R}$ and $D=C_H$.\\

Moreover, if $H$ is a distribution function,
which is not degenerated at 0, then $H$ belongs to exactly one of the following four disjoint classes.
\begin{align}
&\textbf{class 1}: \; H(x) = \left\{ \begin{array}{l@{\quad,\quad}l}
                 0 & x<0\\ \Phi_\sigma\Big(c x^\alpha\Big) & x > 0
               \end{array} \right. \quad (c, \alpha>0). \label{H1}\\
&\textbf{class 2}: \; H(x) = \left\{ \begin{array}{l@{\quad,\quad}l}
                 \Phi_\sigma(-c |x|^\alpha) & x < 0\\ 1 & x > 0
               \end{array} \right. \quad  (c, \alpha >0).\label{H2}\\
&\textbf{class 3}: \; H(x) = \left\{ \begin{array}{l@{\quad,\quad}l}
                 \Phi_\sigma(-c|x|^\alpha) & x < 0\\  \Phi_\sigma(d x^\alpha) & x > 0
               \end{array} \right. \quad (c, d, \alpha >0).\label{H3}\\
&\textbf{class 4}: \; H(x) = \left\{ \begin{array}{l@{\quad,\quad}l}
                 0 & x<-c_1\\ \frac{1}{2} & -c_1 < x < c_2 \\ 1 & x > c_2
               \end{array} \right. \quad (c_1,c_2 \ge 0, \max\{c_1,c_2\}>0). \label{H4}
\end{align}
\end{theorem}

If $Y$ is a real random variable with distribution function $H$, then the convergence in (\ref{dconv}) is equivalent to the distributional convergence \begin{equation} \label{rv}
\frac{m_n-m}{a_n} \stackrel{\mathcal{D}}{\rightarrow} Y.
\end{equation}
Notice that our Theorem \ref{thm1} is the source of a very rich asymptotic theory with unusual normalizing sequences $(a_n)$ going far beyond the square root asymptotic normality.
And even if $H$ is only a sub distribution function, then the limit $Y$ above is an extended random variable with values in $\overline{\mathbb{R}}= \mathbb{R} \cup \{-\infty,\infty\}=[-\infty,\infty]$ and the distributional convergence takes place in the (metric) space $\overline{\mathbb{R}}$.\\

In the literature estimators like $m_n$ are called $M$-estimators. Apparently the first rigorous analysis of such estimators in case $l=1$ goes back
to Huber \cite{Huber1}, who considers $S=\mathbb{R}$ and $h(x,t)=\phi(t-x)$ with $\phi:\mathbb{R} \rightarrow \mathbb{R}$ a convex function.
Br{\o}ns et. al. \cite{Brons} generalize Huber's idea to so-called $\psi$-means, where $\psi$ corresponds to $D^+h$. Both authors show square root asymptotic normality. In a next step Habermann \cite{Habermann}, Niemiro \cite{Niemiro} and Hjort and Pollard \cite{Hjort} consider more general spaces $S$ and extend the approach form the real line to the euclidian space $\mathbb{R}^d$. Here, too, the authors show convergence in distribution to a (multivariate) normal law with rate $\sqrt{n}$. Finally, Bose \cite{Bose} extends these results to the case $l \ge 2$. \\

The paper is organized as follows: In section 2 we prove Theorem \ref{thm1}.  Section 3 is dedicated to characterizing the domains of attraction of the four possible classes. This means that -alternatively to (\ref{nscond})- other equivalent conditions are specified to ensure that the convergence (\ref{dconv}) holds
and in addition it is indicated to which class $H$ belongs. In both sections basic ideas of Smirnov \cite{Smirnov} are used. In section 4 we extend our results for the smallest minimizer to general minimizing points of $U_n$. Section 5 is concerned with smooth derivatives. Namely, if $V=D^+U$ has left-and right derivatives at point $m$, then this results in a square root asymptotic with limits of normal type. Section 6 deals with the huge class of estimators based on mappings $h$ of the shape
$h(\textbf{x},t)=\phi(t-k(\textbf{x}))$, where $\phi:\mathbb{R} \rightarrow \mathbb{R}$ is a convex function and $k:S^l \rightarrow \mathbb{R}$
is any symmetric and measurable map. The special choice  $\phi(t)=t (1_{\{t \ge 0\}}-\alpha)$ and $k(x)=x$ yields Smirnov's \cite{Smirnov} theory for $\alpha$-quantile estimators, $\alpha \in (0,1)$. In fact, a good deal more examples are presented and discussed. The last section 7 (Appendix)
contains several technical results, which are used in our proofs.

\section{Proof of Theorem \ref{thm1}}
The following observation is fundamental:
\begin{equation} \label{eq1}
 H_n(a_n x+m) = \mathbb{P}(m_n \le m+a_n x)=\mathbb{P}(D^+U_n(m+a_n x) \ge 0),
\end{equation}
where the last equality follows from Theorem 1.1 in Ferger \cite{Ferger1}.

For $\textbf{i}=(i_1,\ldots,i_l)$ let $\textbf{X}_\textbf{i}:=(X_{i_1},\ldots,X_{i_l})$. Then
$$D^+U_n(m+a_n x)= \binom{n}{l}^{-1} \sum_{\textbf{i}} D^+h(\textbf{X}_\textbf{i},m+a_n x),$$ where the summation $\sum_\textbf{i}$ extends over
all $\textbf{i}$ with $1\le i_1<i_2<\ldots<i_l\le n$. Here the number of summands is equal to $\binom{n}{l}$, whence
\begin{eqnarray}
 & &\{D^+U_n(m+a_n x) \ge 0\} \nonumber\\
 &=&\{\binom{n}{l}^{-1}\sum_{\textbf{i}} (D^+h(\textbf{X}_\textbf{i},m+a_n x)-D^+U(m+a_n x))\ge -D^+U(m+a_n x)\}\nonumber \\
 &=&\{-\sqrt{n}\binom{n}{l}^{-1}\sum_{\textbf{i}} (D^+h(\textbf{X}_\textbf{i},m+a_n x)-D^+U(m+a_n x)) \le \sqrt{n} D^+U(m+a_n x)\} \nonumber\\
 &=&\{-L_n(x) \le \delta_n(x)\}, \label{eq2}
\end{eqnarray}
where $$L_n(x)=\sqrt{n} \binom{n}{l}^{-1}\sum_{\textbf{i}} (D^+h(\textbf{X}_\textbf{i},m+a_n x)-D^+U(m+a_n x)).$$
By inserting the term $D^+h(\textbf{X}_\textbf{i},m)$
it follows that
\begin{eqnarray}
 & &L_n(x) = \sqrt{n} \binom{n}{l}^{-1}\sum_{\textbf{i}}D^+h(\textbf{X}_\textbf{i},m) \nonumber\\
        &+& \sqrt{n} \binom{n}{l}^{-1}\sum_{\textbf{i}} (D^+h(\textbf{X}_\textbf{i},m+a_n x)-D^+h(\textbf{X}_\textbf{i},m)-D^+U(m+a_n x)). \label{repL}
\end{eqnarray}
One can easily see from the definition that $D^+h(\textbf{x},m)$ is symmetric in $\textbf{x}$. Moreover, $\mathbb{E}[D^+h(\textbf{X}_\textbf{i},m)]=D^+U(m)=0$ by (\ref{Dpm}) and assumption (A1). Thus the first summand
in the decomposition (\ref{repL}) is a normalized and centered $U$-statistic with kernel $D^+h(\textbf{x},m)$. Consequently, (A3) allows us to apply the Central Limit Theorem for U-statistics, confer, e.g., Proposition 4.2.1 in Koroljuk and Borovskich \cite{Koroljuk}, which yields:
\begin{equation} \label{CLTU}
 \sqrt{n} \binom{n}{l}^{-1}\sum_{\textbf{i}}D^+h(\textbf{X}_\textbf{i},m) \stackrel{\mathcal{D}}{\longrightarrow} N(0, \sigma^2).
\end{equation}
Let the second summand in the decomposition (\ref{repL}) be denoted by $N_n(x)$, i.e.
\begin{equation} \label{neg}
 N_n(x):=\sqrt{n} \binom{n}{l}^{-1}\sum_{\textbf{i}} (D^+h(\textbf{X}_\textbf{i},m+a_n x)-D^+h(\textbf{X}_\textbf{i},m)-D^+U(m+a_n x)).
\end{equation}
Another application of (\ref{Dpm}) and (A1) gives that $N_n(x)$ is also a centered and normalized $U$-statistic
with symmetric kernel $$K_n(\textbf{x})\equiv K_{n,x}(\textbf{x})= D^+h(\textbf{x},m+a_n x)-D^+h(\textbf{x},m)-D^+U(m+a_n x).$$
By Lemma A on p.183 in Serfling \cite{Serfling} we have that for all $n \in \mathbb{N}$ and every $x \in \mathbb{R}$:
\begin{eqnarray}
Var[N_n(x)] &\le& l \;  Var[K_n(\textbf{X})] \nonumber\\
            &=& Var[D^+h(\textbf{X},m+a_n x)-D^+h(\textbf{X},m)] \nonumber\\
            &\le& \mathbb{E}[(D^+h(\textbf{X},m+a_n x)-D^+h(\textbf{X},m))^2] \label{N}
\end{eqnarray}
If $x=0$, then $N_n(x)=N_n(0)=0$. For the case $x \neq 0$ we will use the subsequence criterion and that every positive null sequence contains a subsequence that converges from above to zero.
Assume that $x >0$. Let $(n_k)_{k \in \mathbb{N}}$ be a subsequence of the natural numbers. Then $(a_{n_k})_{k \in \mathbb{N}}$ contains a subsequence $(a_{n_{k_l}})_{l \in \mathbb{N}}$ such that
$a_{n_{k_l}} \downarrow 0$ as $l \rightarrow \infty$. Let us write $b_l=a_{n_{k_l}}$ for short. Since $h(\textbf{x}, \cdot)$ is a convex function, it follows that
$D^+h(\textbf{x}, \cdot)$ is an increasing and right-continuous function on $\mathbb{R}$ for each $\textbf{x} \in S^l$, confer Theorem 24.1 in Rockafellar \cite{Rockafellar}.  Therefore $Z_l:=D^+h(\textbf{X},m+b_l x)-D^+h(\textbf{X},m) \ge 0$ and because of $m+b_l x \downarrow m$ the squares $Z_l^2=(D^+h(\textbf{X},m+b_l x)-D^+h(\textbf{X},m))^2 \downarrow 0, l \rightarrow \infty$. According to (\ref{incr}) of Lemma \ref{LA1} in the appendix  the $Z_l^2$ are
$\mathbb{P}$-integrable for eventually all $l \in \mathbb{N}$. Thus an application of the Monotone Convergence Theorem yields that
\begin{equation} \label{ub}
\mathbb{E}[(D^+h(\textbf{X},m+b_l x)-D^+h(\textbf{X},m))^2] \rightarrow 0, l \rightarrow \infty.
\end{equation}
Conclude from (\ref{N}) that $Var[N_{n_{k_l}}(x)] \rightarrow 0$, whence by the subsequence criterion actually $Var[N_n(x)] \rightarrow 0$ as $n \rightarrow \infty$ for all $x>0$.
Next, assume that $x<0$. Then $m+b_l x \uparrow m$, whence $0 \ge Z_l \uparrow D^-h(\emph{X},m)-D^+h(\textbf{X},m)=:Z$ upon noticing that $D^+h(\textbf{x},m-0)=D^-h(\textbf{x},m)$ by Theorem 24.1 in Rockafellar \cite{Rockafellar}. Now, $-Z \ge 0$ and by linearity
$$\mathbb{E}[-Z]= \mathbb{E}[D^+h(\textbf{X},m)]-\mathbb{E}[D^-h(\textbf{X},m)]= D^+U(m)-D^-U(m)=0,$$
where the second last equality holds by (\ref{Dpm}) and the last equality by assumption (A1). Thus $Z=0 \; \mathbb{P}$-a.s. and since $Z_l^2 \downarrow Z^2 \stackrel{a.s.}{=}0$ another application of the Monotone Convergence Theorem shows that (\ref{ub}) is true also for every negative $x$. With the same arguments as for positive $x$ we obtain that $Var[N_n(x)] \rightarrow 0$ as $n \rightarrow \infty$ for all $x<0$.
In summary, $\mathbb{E}[N_n(x)]=0$ for all $n \in \mathbb{N}$ and $Var[N_n(x)] \rightarrow 0, n \rightarrow \infty$ for all $x \in \mathbb{R}$. So, we can infer that
$N_n(x) \stackrel{\mathbb{P}}{\longrightarrow} 0$ for every $x \in \mathbb{R}$. Recall that
$$
 L_n(x) = \sqrt{n} \binom{n}{l}^{-1}\sum_{\textbf{i}}D^+h(\textbf{X}_\textbf{i},m)+N_n(x) \quad \forall \; x \in \mathbb{R}.
$$
Hence (\ref{CLTU}) and Slutsky's theorem ensure that
\begin{equation} \label{dc}
 L_n(x) \stackrel{\mathcal{D}}{\longrightarrow} N(0,\sigma^2) \quad \forall \; x \in \mathbb{R}.
\end{equation}

The limit law (\ref{dc}) makes it easy to prove the following result.\\

\begin{theorem} \label{thm2} If (A0)-(A4) are satisfied, then
$$
 R_n(x):=H_n(a_n x+m)-\Phi_\sigma(\delta_n(x)) \rightarrow 0 \quad \forall \; x \in \mathbb{R}.
$$
\end{theorem}

\begin{proof} By equations (\ref{eq1}) and (\ref{eq2}) we have that
\begin{equation} \label{eq3}
H_n(a_n x+m)=\mathbb{P}(-L_n(x) \le \delta_n(x))=G_{n,x}(\delta_n(x))
\end{equation}
for all $x \in \mathbb{R}$, where $G_{n,x}$ denotes the distribution function of $-L_n(x)$. It follows from (\ref{dc}) that
$G_{n,x}(u) \rightarrow \Phi_\sigma(u), n \rightarrow \infty,$ for every $u \in \mathbb{R}$. By P\'{o}lya's theorem one actually has
that $||G_{n,x}-\Phi_\sigma||:=\sup_{u \in \mathbb{R}}|G_{n,x}(u)-\Phi_\sigma(u)| \rightarrow 0$.
Since by (\ref{eq3})
$$
 |R_n(x)|=|G_{n,x}(\delta_n(x))-\Phi_\sigma(\delta_n(x))|\le ||G_{n,x}-\Phi_\sigma||,
$$
the assertion follows.
\end{proof}

From Theorem \ref{thm2}, in turn, one easily obtains the equivalence of (1) and (2) in Theorem \ref{thm1} and the relation between $H$ and $\delta$. We call this \emph{the first part of Theorem \ref{thm1}}.
Later on this knowledge is used to prove \emph{the second part}.\\

\begin{proof} (First part of Theorem \ref{thm1})  Assume (1) holds, where w.l.o.g. the function is rcll, because otherwise we can replace $H$ with $\tilde{H}$, confer the discussion in front of Theorem \ref{thm1}. Then by Theorem \ref{thm2}
$$
\Phi_\sigma(\delta_n(x))=(H_n(a_n x+m)-H(x))-R_n(x)+H(x) \rightarrow H(x) \quad \forall \; x \in C_H.
$$
Since $\Phi_\sigma$ is invertible, it follows that $\delta_n(x) \rightarrow \Phi_\sigma^{-1}(H(x))$ for every $x \in C_H$.
Thus (2) holds with $\delta(x)=\Phi_\sigma^{-1}(H(x)), x \in \mathbb{R,}$ and $D=C_H$, which is known to be dense in $\mathbb{R}$.

Conversely, assume that (2) is true. Then $H_n(a_n x+m)=R_n(x)+\Phi_\sigma(\delta_n(x)) \rightarrow \Phi_\sigma(\delta(x))$ for all $x \in D$ by Theorem \ref{thm2}.
Now, the functions $G_n$ given by $G_n(x):=H_n(a_n x+m), x \in \mathbb{R},$ are increasing for every $n \in \mathbb{N}$, because  the constants $a_n$ are positive. This also applies to $H$ with
$H(x)=\Phi_\sigma(\delta(x))$. To see this first note that the functions $\delta_n$ are increasing, because $V=D^+U$ is increasing.
Therefore, the limit function $\delta$ is increasing as well. This shows that $H$ is increasing. So, we have that
$G_n(x) \rightarrow H(x)$ for all $x$ in a dense subset $D$ and that all involved functions are increasing. Then Lemma 5.74 in Witting and M\"{u}ller-Funk \cite{Witting}
says that the convergence holds for all $x \in C_H$, which is (1). (Notice that in that Lemma 5.74 it is required that the set $D$ is also denumerable. However, checking the (short) proof shows that countability is not needed.)
\end{proof}

Note that under (A1), for all $n \in \mathbb{N}, \delta_n(x) \le 0$ and $\delta_n(x) \ge 0$ according as $x \le 0$ or $x \ge 0$. %Moreover under (A1), $\delta_n(0)=0$ for each $n \in \mathbb{N}$.
Thus  by taking the limit $n \rightarrow \infty$ it follows that $\delta:\mathbb{R} \rightarrow [-\infty,\infty]$ has the following properties: it is increasing (as we saw in the proof above), non-positive on $(-\infty,0]$,
non-negative on $[0,\infty)$ and in particular $\delta(0)=0$. In the following we will show that $\delta$ can indeed have only four different forms, see (\ref{d1})-(\ref{d4}) below. We take the first step on the way there with the following result:\\

\begin{proposition} \label{funceq} Assume that in Theorem \ref{thm1} the limit $H$ in (1) is a distribution function not degenerated at zero. Then the limit $\delta$ in (2) satisfies the following functional equations:
\begin{equation} \label{fe}
  \delta(x)= \sqrt{k}\; \delta(\alpha_k x) \quad \forall \; x \in \mathbb{R} \quad \forall \; k \in \mathbb{N},
\end{equation}
where the $\alpha_k$ are positive constants given by $\alpha_k = \lim_{n \rightarrow \infty} \frac{a_{nk}}{a_n}$.
\end{proposition}

\begin{proof} It should be remembered that the first part of Theorem \ref{thm1} has already been proved and therefore can be used in what follows. For each fixed $k \in \mathbb{N}, (nk)_{n \in \mathbb{N}}$ is a subsequence of the natural numbers. From (2) of Theorem \ref{thm1} it follows that
$\delta_{nk}(x)=\sqrt{n k} V(m+ a_{n k} x) \rightarrow \delta(x)$ for every $x \in C_H$, whence
\begin{equation*}
\sqrt{n} V(m+ a_{n k} x) \rightarrow \frac{\delta(x)}{\sqrt{k}}=: \delta^{(k)}(x)  \quad \forall \; x \in C_H.
\end{equation*}
Since $a_{nk} \rightarrow 0, n \rightarrow \infty$, another application of Theorem \ref{thm1} (first part) yields that
\begin{equation} \label{c1}
 H_n(a_{nk}x+m) \rightarrow \Phi_\sigma(\delta^{(k)}(x)) =:H^{(k)}(x)  \quad \forall \; x \in C_H = C_{H^{(k)}}.
\end{equation}
Here, the equality of the continuity-sets follows from $C_\delta=C_{\delta^{(k)}}$.
By the convergence (\ref{c1}) and by the convergence in (1) of Theorem \ref{thm1} the Convergence of Strict Types (confer Theorem 25, p.265 in Fristedt and Gray \cite{Fristedt}) says that $\alpha_k = \lim_{n \rightarrow \infty} \frac{a_{nk}}{a_n}>0$ and that $H^{(k)}(x)=H(\alpha_k x)$ for all $x \in \mathbb{R}$. Using the identity $H(x)= \Phi_\sigma(\delta(x)), x \in \mathbb{R},$ from Theorem \ref{thm1} we obtain that $\Phi_\sigma(\delta(x)/\sqrt{k})=\Phi_\sigma(\delta(\alpha_k x)), x \in \mathbb{R}$, which finally results in
$\delta(x)= \sqrt{k}\; \delta(\alpha_k x) \text{ for all } x \in \mathbb{R}$,
because $\Phi_\sigma$ is invertible. %Recall that $\delta(x)= \Phi_\sigma^{-1}(H(x)), x \in \mathbb{R},$ and that $H$ is right continuous, whence $\delta$ is right continuous as well.
%Since $C_H$ lies dense in $\mathbb{R}$, the functional equation (\ref{fe}) follows from (\ref{fe0}).
\end{proof}

Note that $\alpha_1=1$ no matter how the sequence $(a_n)_{n \in \mathbb{N}}$ looks like.
Moreover, the functional equation (\ref{fe}) reduces to $\delta(x)=\delta(x)$, a tautology.
In the analysis of the functional equations (\ref{fe}) we distinguish 3 possible cases. \\

\textbf{Case 0}: Suppose there exists some natural $k \ge 2$ such that $\alpha_k > 1$.
\begin{itemize}
\item[(i)] Let $x > 0.$ Then $\alpha_k x > x$ and thus $\delta(\alpha_k x) \ge \delta(x)$ by monotonicity of $\delta$. As a consequence
\begin{equation} \label{or}
 \delta(x) \le 0 \quad \text{or} \quad \delta(x)=\infty,
\end{equation}
 because otherwise $0 < \delta(x) < \infty$ and therefore $1 \le \frac{\delta(\alpha_k x)}{\delta(x)} = \frac{1}{\sqrt{k}}$, where the last equality follows from (\ref{fe}) by division with $\delta(x) \in (0,\infty)$. It follows that $\sqrt{k} \le 1$ in contradiction to $k \ge 2$. Next, we will show that $\delta(x) > 0$ for every $x >0$. In fact, assume that there is some $x_0>0$ such that
    \begin{equation} \label{le0}
    \delta(x_0) \le 0.
    \end{equation}
    Then actually $\delta(x) \le 0$ for all $x>0$. To see this fix some $x>0$ and
    put
    $$y_l:=\alpha_k^l x_0, \;\; l \in \mathbb{N}_0.$$ Since $\alpha_k >1$ and $x_0$ is positive, it follows that $y_l \uparrow \infty$ as $ l \rightarrow \infty$. In particular, $y_l \ge x$ for some sufficiently large $l$. Consequently,
    \begin{equation} \label{yl}
    \delta(x) \le \delta(y_l).
    \end{equation}
    But
    \begin{equation*}
     \delta(y_l)=\delta(\alpha_k \alpha_k^{l-1} x_0)=\delta(\alpha_k y_{l-1})=k^{-1/2} \delta(y_{l-1})=\ldots=k^{-l/2} \delta(y_0)=k^{-l/2}\delta(x_0).
    \end{equation*}
    Note that this formula applies to every $\alpha_k>0$ and every $x_0 \in \mathbb{R}$. For later use we write down the basic equality:
    \begin{equation} \label{solrecur}
     \delta(y_l)=k^{-l/2} \delta(x_0).
    \end{equation}
    Now it follows from (\ref{yl}),(\ref{solrecur}) and (\ref{le0}) that $\delta(x) \le 0$ for all $x>0$ and therefore $H(x)=\Phi_\sigma(\delta(x))\le1/2$ for all $x > 0$ in contradiction to $H$ is a distribution function.
    Thus we have shown that $\delta(x) >0$ for all $x>0$, which by (\ref{or}) ensures that
    \begin{equation} \label{inf}
    \delta(x)= \infty \quad \forall \; x>0.
    \end{equation}
\item[(ii)] Let $x<0$. This case can be treated analogously as in (i) above. One obtains:
    \begin{equation} \label{minf}
    \delta(x)= -\infty \quad \forall \; x<0.
    \end{equation}
\end{itemize}
Since $H(x)= \Phi_\sigma(\delta(x))$ it follows from (\ref{inf}) and (\ref{minf}) that $H(x)=1$ for all $x>0$ and $H(x)=0$ for all $x<0$, which means that $H$ is a distribution function degenerated at zero in contrast to our assumption. This means that \textbf{Case 0 cannot occur}.\\

\textbf{Case 1}: There exists some natural $k \ge 2$ such that $\alpha_k <1$.
\begin{itemize}
\item[(i)] Let $x < 0.$ Then $\alpha_k x > x$ and by (\ref{fe}),  $\frac{1}{\sqrt{k}} \delta(x)=\delta(\alpha_k x) \ge \delta(x)$. As a consequence
$(+) \;\; \delta(x) \ge \sqrt{k} \delta(x)$. We obtain that
\begin{equation*} %\label{or3}
 \delta(x) \le 0 \quad \text{or} \quad \delta(x)=\infty,
\end{equation*}
because otherwise $0 < \delta(x) < \infty$ and therefore $1 \ge \sqrt{k}$ by division in (+), a contradiction to $k \ge 2$.
The case $\delta(x)=\infty$ is impossible for each $x<0$, because $\delta(x)\le \delta(0)=0< \infty$. So it remains the case
$(a) \;\;\delta(x) \le 0 \quad \forall \; x<0$. Here we make a further case distinction, namely:
\begin{itemize}
\item[(a1)] $\delta(x) > -\infty \quad \forall \; x<0.$
\item[(a2)] There exists some $x_0<0$ such that $\delta(x_0)=-\infty.$ We claim that then $\delta(x)=-\infty$ actually for all $x<0$.
To see this assume that there is some $x<0$ with $\delta(x)> -\infty$. Put $y_l:=\alpha_k^l x_0, l \in \mathbb{N}_0$. Since $\alpha_k \in (0,1)$ and
$x_0<0$, it follows that $y_l \uparrow 0, l \rightarrow \infty$, whence $y_l \ge x$ for some $l$ sufficiently large. An application of (\ref{solrecur}) gives that
$\delta(x_0)=k^{l/2} \delta(y_l) \ge k^{l/2} \delta(x) > -\infty$, so that $\delta(x_0)> -\infty$ as well in contradiction to our assumption.
\end{itemize}
So with a view to (a), only the two cases remain in total:
\begin{itemize}
\item[(a1)] $-\infty < \delta(x) \le 0 \quad \forall \; x<0.$
\item[(a2)] $\delta(x)=-\infty \quad \forall \; x<0.$
\end{itemize}
Similarly, for $x>0$ you get the only possible two cases:
\begin{itemize}
\item[(a3)] $0 \le \delta(x) <\infty \quad \forall \; x>0.$
\item[(a4)] $\delta(x)=\infty \quad \forall \; x>0.$
\end{itemize}
Of the 4 possible combinations, (a2) and (a4) are ruled out because the associated function $H(x)=\Phi_\sigma(\delta(x))$ is the same as in Case 0.
This leaves the following 3 combinations:
\begin{itemize}
\item[(1)] $\delta(x)=-\infty$ if $x<0$ and $0\le \delta(x)< \infty$  if  $x>0$.
\item[(2)] $-\infty < \delta(x) \le 0$ if $x<0$ and $\delta(x)= \infty$  if  $x>0$.
\item[(3)] $-\infty <\delta(x) \le 0$ if $x<0$ and $0 \le \delta(x)< \infty$  if  $x>0$.
\end{itemize}
Now the solutions of the functional equations (\ref{fe}) under the respective constraints (1)-(3)are:
\begin{align}
&\textbf{(1)}: \; \delta(x) = \left\{ \begin{array}{l@{\quad,\quad}l}
                 -\infty & x<0\\ c x^\alpha & x > 0
               \end{array} \right. \quad (c, \alpha>0). \label{d1}\\
&\textbf{(2)}: \; \delta(x) = \left\{ \begin{array}{l@{\quad,\quad}l}
                 -c |x|^\alpha & x < 0\\ \infty & x > 0
               \end{array} \right. \quad  (c, \alpha >0).\label{d2}\\
&\textbf{(3)}: \; \delta(x) = \left\{ \begin{array}{l@{\quad,\quad}l}
                 -c|x|^\alpha & x < 0\\  d x^\alpha & x > 0
               \end{array} \right. \quad (c, d, \alpha >0).\label{d3}
\end{align}
These solutions can be found in Smirnov \cite{Smirnov} on pp.105-106, who in turn refers to Gnedenko \cite{Gnedenko}.\\
\end{itemize}

\textbf{Case 2}: There exists some integer $k \ge 2$ such that $\alpha_k=1$. Then $(*) \; \delta(x)=\sqrt{k} \delta(x)$ for all $x \in \mathbb{R}.$
Assume that $\delta(x) \in \mathbb{R} \setminus \{0\}$. Then division by $\delta(x)$ in $(*)$ yields $1=\sqrt{k}$, in contradiction to $k \ge 2$.
Thus $\delta(x) \in \{-\infty, 0, \infty\}$. Recall that $\delta(0)=0$ and that $\delta$ is increasing.
It follows that:
\begin{equation}
\textbf{(4)}: \; \delta(x) = \left\{ \begin{array}{l@{\quad,\quad}l}
                 -\infty & x<-c_1\\ 0 & -c_1 < x < c_2\\ \infty & x > c_2
               \end{array} \right. \quad (c_1,c_2 \ge 0, \max\{c_1,c_2\}>0). \label{d4}
\end{equation}
Here, $c_1, c_2 \ge 0$ results from $\delta(0)=0$. Finally, $H=\Phi_\sigma \circ \delta$ by Theorem \ref{thm1} (first part), which is degenerated at zero if $c_1=0=c_2$. This is a contradiction to
our assumption and thus $\max\{c_1,c_2\}>0$.\\

From (\ref{d1})-(\ref{d4}) we immediately obtain the four possible classes (\ref{H1})-(\ref{H4}) by another application of Theorem \ref{thm1} (first part). This completes our proof of Theorem \ref{thm1}.

\section{Domains of attraction of the limits $H$}
Now that we have specified all possible limit distributions, the next step is to develop general conditions on a particular pair $(P,h)$ such that (\ref{dconv}) holds with limit $H$ a distribution function not degenerated at $0$.
In that case $(P,h)$ is said to be in the \emph{domain of attraction of} $H$, symbolically: $(P,h) \in \mathcal{D}(H)$. More specifically, we would like to show how to find constants $a_n$ so that $(m_n-m)/a_n$  converges in distribution, as well as determining the limit itself.

According to (\ref{Dpm}) we have that
\begin{equation} \label{V}
 V(t)= D^+U(t)= \int_{S^l} D^+ h(\textbf{x},t) P^l(d\textbf{x})= \mathbb{E}[D^+ h(\textbf{X},t)], \quad t \in \mathbb{R}.
\end{equation}
Recall that Condition (A1) has the form $V(m)=0=V(m-)$. Notice that here $m$ need not to be uniquely determined. However, if this is the case, then there is an equivalent characterisation:\\

\begin{lemma} \label{unique} Suppose $V(m)=0=V(m-)$. Then $m$ is unique if and only if
\begin{equation} \label{Vu}
 V(x)<0 \quad \forall \; x<m \quad \text{ and } \quad V(x)>0 \quad \forall \; x>m.
\end{equation}
\end{lemma}

\begin{proof} Assume that $m$ is unique, but there exists some $x<m$ such that $V(x) \ge 0$ or there exists some $y>m$ such that $V(y) \le 0$. In the first case it follows that
$0 \le V(x) \le V(m) =0$, because $V$ is increasing. Thus for (every) $z \in (x,m)$ we know that $V(z)=0$, whence $V(z-)=0$ as well. This means that $z$ satisfies (A1), which contradicts the
uniqueness of $m$. In the second case one can argue in the same way.

Conversely, assume that (\ref{Vu}) holds, but there exists some $m^\prime \neq m$ such that $V(m^\prime)=0=V(m^\prime-)$. However, $V(m^\prime)=0$ immediately contradicts (\ref{Vu}).
\end{proof}

The following theorems give necessary and sufficient conditions for $(P,h) \in \mathcal{D}(H)$, where $H$ runs through the \textbf{classes 1-4}.\\

\begin{theorem} \label{class1} Assume that conditions (A1)-(A4) are satisfied. Then in order that $(P,h) \in \mathcal{D}(H)$ with $H \in $ \textbf{class 1} it is necessary and sufficient that $m$ is unique,
\begin{equation} \label{a1}
\frac{V(m+t)}{V(m-t)} \rightarrow 0, \quad t \downarrow 0,
\end{equation}
and
\begin{equation} \label{b1}
\frac{V(m+\tau t)}{V(m+t)} \rightarrow \tau^\alpha, \quad t \downarrow 0, \quad \forall \; \tau>0.
\end{equation}
A possible choice for the constants in (\ref{dconv}) is $a_n:= V^{-1}(1/\sqrt{n})-m \downarrow 0$.
In this case, $\alpha$ in (\ref{b1}) and $c=1$ are the parameters in (\ref{H1}).
\end{theorem}

\begin{proof} Sufficiency: Recall that $V$ is increasing and right continuous with left limits (rcll). By (\ref{Vu}) the function $V$ is not flat at $m$ and therefore $V^{-1}$ is continuous at $0$ by Lemma 1.18 in Witting \cite{Witting0}. Infer from that continuity and the definition of $(a_n)$ that $a_n \downarrow 0$.
Moreover, since $m+a_n=V^{-1}(1/\sqrt{n})$, it follows that
\begin{equation} \label{Vm0}
V(m+a_n-)=V(V^{-1}(1/\sqrt{n})-) \le \frac{1}{\sqrt{n}} \le V(V^{-1}(1/\sqrt{n}))=V(m+a_n).
\end{equation}
Consequently, we have $\frac{1}{V(m+a_n)} \le \sqrt{n} \le \frac{1}{V(m+a_n-)}$ and therefore
\begin{equation} \label{Vm}
 \frac{V(m+a_nx)}{V(m+a_n)} \le \sqrt{n}V(m+a_n x) \le \frac{V(m+a_nx)}{V(m+a_n-)}.
\end{equation}
Deduce from (\ref{b1}) with $\tau=x >0$ and $t=a_n \downarrow 0$ that
\begin{equation} \label{xa}
 \lim_{n \rightarrow \infty} \frac{V(m+ a_n x)}{V(m+a_n)} = x^\alpha \quad \forall \; x>0.
\end{equation}

Fix some $\epsilon \in [0,1/2]$. Another application of (\ref{b1}) with $\tau = x/(1-\epsilon)$ and $t = a_n(1-\epsilon) \downarrow 0$ yields
\begin{equation} \label{xa2}
 \lim_{n \rightarrow \infty} \frac{V(m+ a_n x)}{V(m+a_n(1-\epsilon))} = \big(\frac{x}{1-\epsilon}\big)^\alpha \quad \forall \; x>0 \quad \forall \; \epsilon \in [0,1/2].
\end{equation}

Put $f_n(\epsilon):= V(m+a_n x)/V(m+a_n(1-\epsilon))$ and $f(\epsilon):= \big(\frac{x}{1-\epsilon}\big)^\alpha$ with $x>0$ fixed. Observe that $f_n$ and $f$ are increasing and
in addition $f$ is continuous. By P\'{o}lya's Theorem the sequence $(f_n)$ converges uniformly on $[0,1/2]$, whence one can interchange the limits, i.e.
$$
 \lim_{\epsilon \downarrow 0} \lim_{n \rightarrow \infty} f_n(\epsilon) = \lim_{n \rightarrow \infty} \lim_{\epsilon \downarrow 0} f_n(\epsilon)=\lim_{n \rightarrow \infty} \frac{V(m+ a_n x)}{V(m+a_n-)}.
$$
Thus by taking the limit $\epsilon \downarrow 0$ in (\ref{xa2}) we arrive at
\begin{equation} \label{xa3}
 \lim_{n \rightarrow \infty} \frac{V(m+ a_n x)}{V(m+a_n-)} = x^\alpha \quad \forall \; x>0.
\end{equation}

From (\ref{Vm}) in combination with (\ref{xa}) and (\ref{xa3}) it follows that
\begin{equation} \label{delta1p}
 \delta_n(x)= \sqrt{n} V(m + a_n x) \rightarrow x^\alpha \quad \forall \; x>0.
\end{equation}

Next, consider $x<0$. Then $V(m+a_n x)<0$ by (\ref{Vu}) and $\sqrt{n} \ge 1/V(m+a_n)$ by (\ref{Vm0}) and therefore
\begin{equation} \label{delta1n}
 \sqrt{n} V(m+a_n x) \le \frac{V(m+a_n x)}{V(m+a_n)}= \frac{V(m+a_n x)}{V(m-a_n x)} \cdot \frac{V(m-a_n x)}{V(m+a_n)}.
\end{equation}
Here, on the right side the first factor is negative for all $n \in \mathbb{N}$ and its inverse converges to zero by condition (\ref{a1}) with $t =-a_n x \downarrow 0$.
Thus the first factor converges to $-\infty$, whereas the second factor converges to $(-x)^\alpha>0$ by condition (\ref{b1}) with $\tau=-x >0$ and $t=a_n \downarrow 0$.
So, the upper bound in (\ref{delta1n}) converges to $-\infty$, which finally shows that
\begin{equation} \label{delta1n2}
\delta_n(x) = \sqrt{n} V(m+a_nx) \rightarrow -\infty \quad \forall \; x<0.
\end{equation}
Infer from (\ref{delta1p}) and (\ref{delta1n2}) that condition (2) of Theorem \ref{thm1} is fulfilled, where $\delta$ is as in (\ref{d1}) with $c=1$ and $D=\mathbb{R} \setminus \{0\}$. Consequently (\ref{dconv}) follows
with $H \in$ \textbf{class 1}.\\

Necessity: By assumption and Theorem \ref{thm1} there exists some sequence $a_n \rightarrow 0$ such that
\begin{equation} \label{ul}
\delta_n(x)=\sqrt{n} V(m+a_nx) \rightarrow \delta(x) \quad \forall \; x \in C_H = \mathbb{R} \setminus \{0\}
\end{equation}
with $\delta$ as in (\ref{d1}).
Assume that $m$ is not unique, i.e. $V(m^\prime)=0$ for some $m^\prime \neq m$. If $m^\prime <m$, then by monotonicity $V(u)=0$ for every $u \in [m^\prime,m]$. It follows that for $x<0$, $m+a_n x \in [m^\prime,m]$
and thus $\delta_n(x)= \sqrt{n} V(m + a_n x)=0$ for eventually all $n \in \mathbb{N}$, in contradiction to $\delta_n(x) \rightarrow \delta(x)=-\infty$. If $m^\prime > m$, then one obtains in the same way that for $x>0$,
$\delta_n(x)=0$ for eventually all $n \in \mathbb{N}$, in contradiction to $\delta_n(x) \rightarrow \delta(x)= c x^\alpha >0$. This shows that $m$ is unique.

For the derivation of (\ref{a1}) consider at first an arbitrary fixed $t>0$. W.l.o.g. we may assume that $(a_n)$ is decreasing because otherwise extract a decreasing subsequence. Then there exists some $n=n(t) \in \mathbb{N}$ such that $a_{n+1} \le t \le a_n$ and therefore
\begin{equation} \label{nt}
V(m+a_{n+1}) \le V(m+t) \le V(m+a_n).
\end{equation}
Notice that $V(m-t)<0< V(m+t)$ by (\ref{Vu}). It follows that
\begin{equation} \label{t}
 \frac{\sqrt{n+1}}{\sqrt{n}} \cdot \frac{\delta_{n}(1)}{\delta_{n+1}(-1)} \le \frac{V(m+t)}{V(m-t)} \le \frac{\sqrt{n}}{\sqrt{n+1}} \cdot \frac{\delta_{n+1}(1)}{\delta_{n}(-1)}.
\end{equation}
Note that $n=n(t) \uparrow \infty$ as $t\downarrow 0$. Clearly the first ratio of the lower bound converges to $1$, whereas by (\ref{ul}) the
numerator of the second ratio converges to $c>0$ and the denominator converges to $-\infty$. Consequently the lower bound converges to zero.
This also applies to the upper bound, whence condition (\ref{a1}) follows.

Moreover, for each $\tau >0$ we obtain
\begin{equation} \label{taut}
\frac{\sqrt{n}}{\sqrt{n+1}} \cdot \frac{\delta_{n+1}(\tau)}{\delta_n(1)} \le \frac{V(m+\tau t)}{V(m+t)} \le \frac{\sqrt{n+1}}{\sqrt{n}} \cdot \frac{\delta_n(\tau)}{\delta_{n+1}(1)}.
\end{equation}
Thus another application of (\ref{ul}) results in (\ref{b1}).
\end{proof}

\begin{theorem} \label{class2} Assume that conditions (A1)-(A4) are satisfied. Then $(P,h) \in \mathcal{D}(H)$ with $H \in $ \textbf{class 2} if and only if $m$ is unique and
\begin{equation} \label{a2}
\frac{V(m-t)}{V(m+t)} \rightarrow 0, \quad t \downarrow 0,
\end{equation}
and
\begin{equation} \label{b2}
\frac{V(m-\tau t)}{V(m-t)} \rightarrow \tau^\alpha, \quad t \downarrow 0, \quad \forall \; \tau>0.
\end{equation}
A possible choice for the constants in (\ref{dconv}) is $a_n:=m-V^{-1}(-1/\sqrt{n}) \downarrow 0$ is.
Then $\alpha$ in (\ref{b2}) is the exponent in (\ref{H2}) and $c=1$ there.
\end{theorem}

\begin{proof} Analogously to the above proof.
\end{proof}

\begin{theorem} \label{class3} Assume that conditions (A1)-(A4) are satisfied. Then $(P,h) \in \mathcal{D}(H)$ with $H \in $ \textbf{class 3} if and only if $m$ is unique and
\begin{equation} \label{a3}
\frac{V(m+t)}{V(m-t)} \rightarrow A<0, \quad t \downarrow 0,
\end{equation}
and
\begin{equation} \label{b3}
\frac{V(m+\tau t)}{V(m+t)} \rightarrow \tau^\alpha, \quad t \downarrow 0, \quad \forall \; \tau>0.
\end{equation}
A possible choice for the constants in (\ref{dconv}) is $a_n:= V^{-1}(1/\sqrt{n})-m \downarrow 0$.
Moreover, $\alpha$ in (\ref{b3}) is the exponent in (\ref{H3}) and the constants there are $c=-1/A>0$ with $A$ from (\ref{a3}) and $d=1$.
\end{theorem}

\begin{proof} If-part: As in the proof of Theorem \ref{class1} one shows that $\delta_n(x)=\sqrt{n} V(m+a_n x) \rightarrow x^\alpha$ for
all $x>0$. Consider $x<0$. Then
$$
 \sqrt{n} V(m+a_n x)= \sqrt{n} V(m+a_n (-x)) \Big(\frac{V(m+a_n(-x))}{V(m+a_n x)}\Big)^{-1}.
$$
Here the first factor converges to $(-x)^\alpha$ and the second factor to $A^{-1}$ by condition (\ref{a3}).

We obtain that
$$
\delta_n(x) \rightarrow \delta(x) = \left\{ \begin{array}{l@{\quad,\quad}l}
                 -c |x|^\alpha & x<0\\  x^\alpha & x > 0,
               \end{array} \right.
$$
where $c=-1/A >0$. Thus Theorem \ref{thm1} ensures that $(P,h) \in \mathcal{D}(H)$ with $H \in $ \textbf{class 3}.

Only-if-part: By Theorem \ref{thm1} it follows that $$
\delta_n(x) \rightarrow \delta(x) = \left\{ \begin{array}{l@{\quad,\quad}l}
                 -c |x|^\alpha & x<0\\  d x^\alpha & x > 0,
               \end{array} \right.
$$
where $c$ and $d$ are positive constants. Therefore uniqueness of $m$ can be shown  as in the proof of Theorem \ref{class1} upon noticing that $\delta(x) \neq 0$ for all $x \neq 0$. Moreover, conditions (\ref{a3}) with $A=-d/c<0$ and (\ref{b3}) follow from (\ref{t}) and (\ref{taut}).
\end{proof}

\begin{theorem} \label{class4} Suppose that (A1)-(A4) hold with $m$ being unique. Then $(P,h) \in \mathcal{D}(H)$, where $H \in $ \textbf{class 4} if and only if
\begin{equation} \label{deltastar}
 \sqrt{n} V(a_n^* x+m)  \rightarrow  \delta(x)=\left\{ \begin{array}{l@{\quad,\quad}l}
                 -\infty & x<-c_1\\  0 & -c_1<x<c_2\\ \infty & x>c_2.
               \end{array} \right.
\end{equation}
Here, $c_1, c_2 \ge 0$ with $\max\{c_1,c_2\}>0$ and
$$
 a_n^*=\big(V^{-1}(1/\sqrt{n})-V^{-1}(-1/\sqrt{n})\big)/(c_1+c_2) \downarrow 0.
$$
\end{theorem}

\begin{proof} From the proof of Theorem \ref{class1} we already know that uniqueness of $m$ entails continuity of $V^{-1}$ at $0$. Thus $a_n^*\downarrow 0$.
Further, by Theorem \ref{thm1} the relation $(P,h) \in \mathcal{D}(H)$ with $H \in $ \textbf{class 4} is true if and only if there is some sequence $a_n \rightarrow 0$ such that
\begin{equation} \label{delta4}
 \sqrt{n} V(a_n x+m)  \rightarrow \delta(x).
\end{equation}

Let $x=c_2+\epsilon$ with $\epsilon \in (0,\infty)$. By (\ref{delta4}) there exists some $n_0 = n_0(\epsilon) \in \mathbb{N}$ such that
$\sqrt{n} V(a_n x+m) \ge 1$ and therefore  $V(a_n x+m) \ge 1/\sqrt{n}$ for all $n \ge n_0$. For these $n$ it follows that $a_n x +m \ge V^{-1}(1/\sqrt{n})$, by which
we arrive at $a_n(c_2+\epsilon) \ge V^{-1}(1/\sqrt{n})-m =: \alpha_n$ for all $n \ge n_0$.

Let $x=-c_1+\epsilon$ with $\epsilon \in (0,c_1+c_2)$. Using (\ref{delta4}) we find an integer $n_1=n_1(\epsilon)$ such that $|\sqrt{n} V(a_n x+m)| \le 1$, whence
$V(a_n x+m) \ge -1/\sqrt{n}$ for every $n \ge n_1$. As a consequence one obtain the following equivalent relations:
$a_n x+m \ge V^{-1}(-1/\sqrt{n}) \Leftrightarrow a_n x\ge -m + V^{-1}(-1/\sqrt{n}) \Leftrightarrow a_n(c_1-\epsilon) \le m-V^{-1}(-1/\sqrt{n}) =:\beta_n.$

If $x=c_2-\epsilon, \epsilon \in (0,c_1+c_2)$, then there exists a $n_2 \in \mathbb{N}$ such that for all $n \ge n_2$ the following implications hold:

$|\sqrt{n} V(a_n x+m)| \le 1/2 \Rightarrow  V(a_n x+m) \le \frac{1}{2} \frac{1}{\sqrt{n}} < \frac{1}{\sqrt{n}} \Rightarrow a_n x+m < V^{-1}(1/\sqrt{n})
\Rightarrow a_n(c_2-\epsilon) < \alpha_n.$

Similarly, one treats the case $x=-c_1-\epsilon, \epsilon \in (0,\infty)$. To summarize, we obtain that for each $\epsilon \in (0,c_1+c_2)$ there is some
$N=N(\epsilon) \in \mathbb{N}$ such that
$$
 a_n(c_2-\epsilon)<\alpha_n \le a_n(c_2+\epsilon), \quad a_n(c_1-\epsilon) \le \beta_n < a_n(c_1+\epsilon) \quad \forall \; n \ge N.
$$
If you add both inequalities and then divide by $2 a_n$, the result with $d:=(c_1+c_2)/2$ is
\begin{equation} \label{sum}
 d-\epsilon \le \frac{\alpha_n+\beta_n}{2 a_n} \le d+\epsilon \quad \forall \; n \ge N.
\end{equation}
Define $$\theta_n:=\frac{d}{\frac{\alpha_n+\beta_n}{2 a_n}}-1.$$
Then on the one hand it follows from (\ref{sum}) that $-\frac{\epsilon}{d+\epsilon} \le \theta_n \le \frac{\epsilon}{d-\epsilon}$, whence
$\theta_n \rightarrow 0$.
On the other hand a simple rearrangement of the formula for $\theta_n$ gives $$a_n=\frac{\alpha_n+\beta_n}{2 d}(1+\theta_n)=a_n^*(1+\theta_n).$$
Thus (\ref{delta4}) is the same as
\begin{equation} \label{delta4a}
 \sqrt{n} V(a_n^*(1+\theta_n) x+m)  \rightarrow \delta(x).
\end{equation}

Our proof is complete, once we have shown that (\ref{delta4a}) is equivalent to (\ref{deltastar}). To do this, introduce $\delta_n^*(x):=\sqrt{n} V(a_n^* x+m)$ and  $\delta_n^{**}(x):=\sqrt{n} V(a_n^*(1+\theta_n) x+m).$
Assume that (\ref{deltastar}) holds. Consider $x \in (-c_1,c_2)$. By assumption $\delta_n^*(x) \rightarrow \delta(x)=0$. Observe that $\delta_n^*$ and $\delta$ both are increasing
and that $\delta$ is continuous. Therefore P\'{o}lya's theorem ensures that $\delta_n^*$ converges to $\delta$ uniformly on every compact set $K \subseteq (-c_1,c_2)$.
As a consequence $\delta_n^*(x_n) \rightarrow \delta(x)$ for every sequence $(x_n)$ with $x_n \rightarrow x$. Since $x_n:=(1+\theta_n) x \rightarrow x$, we receive that
$\delta_n^{**}(x)=\delta_n^*(x_n) \rightarrow \delta(x).$ Let $x>c_2$. For some $z \in (c_2,x)$ we have that $x_n \ge z$ for eventually all $n \in \mathbb{N}$. Conclude
from $\delta_n^{**}(x)=\delta_n^*(x_n) \ge \delta_n^*(z) \rightarrow \infty$ that $\delta_n^{**}(x) \rightarrow \infty.$ Similarly, one obtains for $x<-c_1$ that
$\delta_n^{**}(x) \rightarrow -\infty.$ This shows that (\ref{delta4a}) is fulfilled.
For the reverse conclusion notice that $\delta_n^*(x)=\delta_n^{**}(y_n)$ with $y_n=x/(1+\theta_n) \rightarrow x$, so that we can proceed as above.
\end{proof}

\section{Extension to general minimizers of $U_n$}
So far we only considered the smallest minimizing point $m_n$ of $U_n$. Let $m_n^*$ be the largest minimizing point of $U_n$, which by Corollary 3.3 in Ferger \cite{Ferger1} is a real random variable. If $H_n^*(x):=\mathbb{P}(m_n^* <x)$, then
$$
 H_n^*(a_n x+m) = \mathbb{P}(m_n^* < m+a_n x) = \mathbb{P}(D^-U_n(m+a_nx) >0),
$$
where the second equality holds by Theorem 1.1 of Ferger \cite{Ferger1}. (Notice that $H_n^*$ is left continuous with right limits.)
From here on, we can proceed in the same way as in the proof of Theorem \ref{thm1} (first part), replacing $D^+h$ with $D^-h$ in assumptions (A2)-(A4). More precisely, the assumptions are now:

\begin{itemize}
\item[(B0)] = (A0).
\item[(B1)] = (A1).
\item[(B2)] $\mathbb{E}[D^-h(X_1,\ldots,X_l,m)^2] < \infty$.
\item[(B3)] Put $$\zeta^-:= \int_S [\int_{S^{l-1}} D^-h(x_1,x_2,\ldots,x_l,m) d P^{l-1}(x_2,\ldots,x_l)]^2 d P(x_1)$$ for $l\ge 2,$  and
$\zeta^-:=\int_S D^-h(x,m)^2 P(dx)$ for $l=1$. We require that $\zeta^->0$. (Finiteness of $\zeta^-$ follows from (B2).)
\item[(B4)] There exist points $s_0 <0 <s_1$ such that $$\mathbb{E}[\big(D^-h(X_1,\ldots,X_l,m+t)-D^-h(X_1,\ldots,X_l,m)\big)^2] < \infty$$ for $t \in \{s_0,s_1\}$.
\end{itemize}

In complete analogy to the proof of Theorem 1 (part one), we obtain the following:\\

\textbf{Interim Conclusion} Assume that $m$ is a real number such that (B0)-(B4) hold. Then the following statements (1) and (2) are equivalent:
\begin{itemize}
\item[(1)]
\begin{equation} \label{dconv*}
       H_n^*(a_n x+m) \rightarrow H^*(x) \quad \forall \; x \in C_{H^*},
\end{equation}
where $H^*$ is a sub distribution function.
\item[(2)]
\begin{equation} \label{nscond*}
\delta_n^*(x):= \sqrt{n} D^-U(m+ a_n x) \rightarrow \delta^*(x) \in [-\infty,\infty] \quad \forall \; x \in D^*,
\end{equation}
where $D^*$ is some dense subset of $\mathbb{R}$.
\end{itemize}
In both cases, $H^*(x)= \Phi_{\sigma^*}(\delta^*(x))$ and $D^*=C_{H^*}$.
Moreover, ${\sigma^*}^2=l^2 \zeta^-$.\\

Now that we have looked at the smallest and largest minimizer of $U_n$, let's look in general at arbitrary measurable minimizer $\tilde{m}_n \in \text{Argmin}(U_n)$ a.s.
By (\ref{AUn}) every such minimizer is completely characterized by the inequalities
\begin{equation} \label{zeroestim}
 \int_{S^l} D^- h(\textbf{x},\tilde{m}_n) P_n(d\textbf{x}) \le 0 \le \int_{S^l} D^+ h(\textbf{x},\tilde{m}_n) P_n(d\textbf{x}) \;\; \text{a.s.}
\end{equation}

\vspace{0.4cm}
\begin{theorem} Let $\tilde{m}_n$ be a real random variable, which minimizes $U_n$ a.s.
Then all results from section 1 apply verbatim for $\tilde{m}_n$ and its distribution function $\tilde{H}_n(x):=\mathbb{P}(\tilde{m}_n \le x), x \in \mathbb{R}$.
\end{theorem}

\begin{proof} First notice the basic relation
\begin{equation} \label{ineq}
 H_n^* \le \tilde{H}_n \le H_n,
\end{equation}
which holds to be true, because $m_n \le \tilde{m}_n \le m_n^*$ a.s.

We prove that the first part of Theorem \ref{thm1} holds for $m_n$ replaced by $\tilde{m}_n$.
To this end observe that by Lemma \ref{LA2} in the appendix the validity of assumtions (A1)-(A4) entail that of (B1)-(B4).
From Theorem \ref{thm2} it is known that
\begin{equation} \label{Rn}
R_n(x):=H_n(a_n x+m)-\Phi_\sigma(\delta_n(x)) \rightarrow 0 \quad \forall \; x \in \mathbb{R}.
\end{equation}
Since (B1)-(B4) hold, it follows analogously that $$R_n^*(x):=H_n^*(a_n x+m)-\Phi_{\sigma^*}(\delta_n^*(x)) \rightarrow 0
\quad \forall \; x \in \mathbb{R}.$$ By Lemma \ref{LA3} in the appendix, $\delta_n^*(x)=\delta_n(x)$ for all $x \in D_1$ and for all $n \in \mathbb{N}$.
Moreover, $\zeta^-=\zeta$ by Lemma \ref{LA2} in the appendix, whence $\sigma^*=\sigma$. Thus
\begin{equation} \label{Rn*}
R_n^*(x)=H_n^*(a_n x+m)-\Phi_\sigma(\delta_n(x)) \rightarrow 0 \quad \forall \; x \in D_1.
\end{equation}
Put $\tilde{R}_n(x):= \tilde{H}_n(a_n x+m)-\Phi_\sigma(\delta_n(x)), x \in \mathbb{R}$. Then by the equality in (\ref{Rn*}) and by (\ref{ineq})
it follows that $R_n^*(x) \le \tilde{R}_n(x) \le R_n(x)$ for all $x \in D_1$. Thus (\ref{Rn}) and (\ref{Rn*}) make us to apply the sandwich-theorem
resulting in $\tilde{R}_n(x) \rightarrow 0$ for each $x \in D_1$.

Now, assume that
\begin{equation} \label{A}
\tilde{H}_n(a_n x+m) \rightarrow H(x) \quad \forall \; x \in C_H,
\end{equation}
where $H$ is increasing and rcll.
From Lemma \ref{LA3} in the appendix we know that $D_1=A^c$, where $A$ is a countable set. Then $D:=(A \cup D_H)^c=D_1 \cap C_H$
is dense. If follows that
$$
\Phi_\sigma(\delta_n(x))=(\tilde{H}_n(a_n x+m)-H(x))-\tilde{R}_n(x)+H(x) \rightarrow H(x) \quad \forall \; x \in D
$$
and therefore
\begin{equation} \label{B}
\delta_n(x) \rightarrow \delta(x) \quad \forall \; x \in D
\end{equation}
with $\delta(x)=\Phi_\sigma^{-1}(H(x))$ as desired.

Conversely, assume that (\ref{B}) holds.
By Theorem \ref{thm1}, $H_n(a_n x+m) \rightarrow H(x)=\Phi_\sigma(\delta(x))$ for all $x \in C_H$.
According to Lemma \ref{LA3}, $\delta_n^* \rightarrow \delta(x)$ for all $x \in D_2$, where $D_2$ is dense. Recall that $\sigma^*=\sigma$. Thus the Interim Conclusion says that $H^*(a_n x+m) \rightarrow H(x)$ for all $x \in C_{H}$. Hence the inequalities (\ref{ineq}) and another application of the sandwich-theorem yield the convergence (\ref{A}).

As to the second part observe that $\delta$ satisfies the functional equations (\ref{funceq}). This can be shown in the same way as for the sequence $(m_n)$. And the solutions of these have been derived in section 2.
\end{proof}

\section{Square-root asymptotics}
All our conditions on $V$ describing the asymptotic behaviour of the distributions of the estimators $\tilde{m}_n$ are of local type, namely:
How does $V$ behave in a small neighborhood of $m$?
The most simple case is when it is differentiable at $m$ as in the following proposition. Its first part generalizes the results of Huber \cite{Huber1} and of Br{\o}ns et.al. \cite{Brons}. The statement of the second part is new.\\

\begin{proposition} \label{gsmooth} Let (A0)-(A4) be satisfied.
Assume that $V$ is differentiable at $m$ with $V^\prime(m)> 0$.
Then
\begin{equation} \label{asnormV}
 \sqrt{n}(\tilde{m}_n -m) \stackrel{\mathcal{D}}{\rightarrow} N(0,\tau^2)\quad \text{ with } \tau^2=\frac{l^2 \zeta}{V^\prime(m)^2}.
\end{equation}
If $V^\prime(m)=0$, then
\begin{equation} \label{infinity}
 \sqrt{n}(\tilde{m}_n -m) \stackrel{\mathcal{D}}{\rightarrow} Y \quad \text{ in } \; \overline{\mathbb{R}},
\end{equation}
where $\mathbb{P}(Y=-\infty)=\mathbb{P}(Y=\infty)= \frac{1}{2}.$
\end{proposition}

\begin{proof} Let $a_n = 1/\sqrt{n}$. Then by $V(m)=0$
\begin{equation} \label{deltanV}
 \delta_n(x)=\sqrt{n} V(m+ \frac{x}{\sqrt{n}})=\frac{V(m+ \frac{x}{\sqrt{n}})-V(m)}{\frac{x}{\sqrt{n}}} x \quad \forall \; 0 \neq x \in \mathbb{R}.
\end{equation}
Recall that $\delta_n(0)=0$. Therefore, $\delta_n(x) \rightarrow V^\prime(m)\cdot x = \delta(x)$ for all $x \in \mathbb{R}$ and the assertions follow from Theorem \ref{thm1}. Indeed,  $H(x)=\Phi_\sigma(V^\prime(m) x)$ is the distribution function of $N(0, \sigma^2/(V^\prime(m))^2)$, if $V^\prime(m)>0$. This shows (\ref{asnormV}). If $V^\prime(m)=0$, then $H(x)=1/2$ for all $x \in \mathbb{R}$, which is the distribution function of $Y$. This shows the second part (\ref{infinity}).
\end{proof}

In fact, Proposition \ref{gsmooth} admits a slight generalization, which may be useful sometimes.\\

\begin{remark} \label{onesided}
Assume that $V$ has only left- and right derivatives $D^-V(m)$ and $D^+V(m)$ at point $m$, which both are positive. Then it follows from (\ref{deltanV}) and Theorem \ref{thm1} that
$$
 \sqrt{n}(\tilde{m}_n -m) \stackrel{\mathcal{D}}{\rightarrow} Y,
$$
where $Y=1_{\{N<0\}} \frac{\sigma}{D^-V(m)}N+1_{\{N\ge 0\}} \frac{\sigma}{D^+V(m)}N$ and $N=N(0,1)$.
In case that $D^-V(m)=0<D^+V(m)$ the equations (\ref{deltanV}) show that $\delta_n(x) \rightarrow \delta(x)$ for all $x \in \mathbb{R}$, where
$\delta(x)=0$ for $x<0$ and $\delta(x)=D^+V(m) \cdot x$ for $x \ge 0$. Therefore by Theorem \ref{thm1},
\begin{equation} \label{Y}
 \sqrt{n}(\tilde{m}_n -m) \stackrel{\mathcal{D}}{\rightarrow} Y \quad \text{ in } \; \overline{\mathbb{R}}=\mathbb{R} \cup \{-\infty,\infty\},
\end{equation}
where $Y$ has sub distribution function
$H(x)= 1_{\{x<0\}} \frac{1}{2}+ 1_{\{x \ge 0\}} \Phi_\sigma(D^+V(m) x)$.
Notice that $\mathbb{P}(Y=-\infty)=\frac{1}{2}$ and $\mathbb{P}(Y=\infty)=0$.
Moreover, $\mathbb{P}(Y \in (-\infty,0])=0$ and $\mathbb{P}(0 < Y \le x)=\Phi_\sigma(D^+V(m) x)-1/2$.

In case that $D^-V(m)>0=D^+V(m)$ the convergence (\ref{Y}) holds, where $Y$ has sub distribution function
$H(x)= 1_{\{x \ge 0\}} \frac{1}{2}+ 1_{\{x < 0\}} \Phi_\sigma(D^-V(m) x)$. Here, $\mathbb{P}(Y=\infty)=\frac{1}{2}$ and $\mathbb{P}(Y=-\infty)=0$.

Finally, if only $D^+V(m)$ exists and is positive, then
\begin{equation} \label{ThmSerfling1}
 \mathbb{P}(\sqrt{n}(\tilde{m}_n -m)\le x) \rightarrow \Phi_\sigma(D^+V(m) x ) \quad \text{ for all } x \ge 0.
\end{equation}
Similarly, if only $D^-V(m)$ exists and is positive, then
\begin{equation} \label{ThmSerfling2}
 \mathbb{P}(\sqrt{n}(\tilde{m}_n -m)\le x) \rightarrow \Phi_\sigma(D^-V(m) x ) \quad \text{ for all } x \le 0.
\end{equation}
These last two results (\ref{ThmSerfling1}) and (\ref{ThmSerfling2}) follow analogously to the proof
of Theorem \ref{thm1} (first part).
\end{remark}

\vspace{0.3cm}

Next, we think that the basic functions $h$ are generated as follows:
Let $\psi:S^l \times \mathbb{R} \rightarrow \mathbb{R}$ be a bivariate function such that $\psi(\textbf{x},\cdot)$ is monotone increasing for each $\textbf{x} \in S^l$. W.l.o.g. we can assume that $\psi(\textbf{x}, \cdot)$ is rcll, because othererwise use instead its rcll version $\tilde{\psi}(\textbf{x},\cdot)$, compare the section in front of Theorem \ref{thm1}. 
Define
$$
 h(\textbf{x},t) := \int_0^t \psi(\textbf{x},s) ds.
$$
Then $h(\textbf{x}, \cdot)$ is convex for every $\textbf{x} \in S^l$, confer e.g. section 1.6 in Niculescu and Persson \cite{Niculescu}.
Moreover Theorem 24.2 in Rockafellar \cite{Rockafellar} says that $D^+ h(\textbf{x},t) = \psi(\textbf{x},t+)=\psi(\textbf{x},t)$ by right continuity and $D^- h(\textbf{x},t) = \psi(\textbf{x},t-)$. Then $$\zeta= \int_S \mathbb{E}[\psi(x,X_2,\ldots,X_l,m)]^2 P(dx),$$
if $l \ge 2$ and $\zeta= \mathbb{E}[\psi(X_1,m)^2]$, if $l=1$.
From (\ref{V}) it follows that
\begin{equation} \label{Vpsi}
 V(t)= \int_{S^l} \psi(\textbf{x},t) P^l(d\textbf{x})= \mathbb{E}[\psi(\textbf{X},t)], \quad t \in \mathbb{R}.
\end{equation}
Similarly, $V^\pm_n$ in (\ref{Vnpm}) can be rewritten as
$$
V^+_n(t)= \binom{n}{l}^{-1} \sum_{\textbf{i}} \psi(\textbf{X}_\textbf{i},t) \quad \text{ and } \quad V^-_n(t)= \binom{n}{l}^{-1} \sum_{\textbf{i}} \psi(\textbf{X}_\textbf{i},t-).
$$
Here, recall that $\textbf{X}_\textbf{i}=(X_{i_1},\ldots,X_{i_l})$ and that the summation $\sum_\textbf{i}$ extends over all $l$-tuples $\textbf{i}=(i_1,\ldots,i_l)$ with $1\le i_1<i_2<\ldots<i_l\le n$.
The following result is rather just a reformulation of Proposition \ref{gsmooth}. It enables us to compare it with a known result in the literature.\\

\begin{corollary} \label{zero} Let $\tilde{m}_n$ be a real random variable such that
\begin{equation} \label{zeroineq}
 V^-_n(\tilde{m}_n) \le 0 \le V^+_n(\tilde{m}_n) \quad \text{a.s.}
\end{equation}
Suppose that $V$ in (\ref{Vpsi}) is continuous and strictly increasing on $\mathbb{R}$ and differentiable at $m$. Moreover, assume that:
\begin{itemize}
\item[(0)] $\psi(\cdot,t)$ and $\psi(\cdot,t-)$ are $P^l$-integrable for all $t \in \mathbb{R}$.
\item[(1)] $V(m)=0$.
\item[(2)] $\mathbb{E}[\psi(\textbf{X},m)^2] < \infty$.
\item[(3)] $\zeta >0.$
\item[(4)] There exist reals $t_0 < 0 < t_1$ such that $\mathbb{E}[\big(\psi(\textbf{X},m+t)-\psi(\textbf{X},m)\big)^2] < \infty$ for $t \in \{t_0, t_1\}$.
%\item[(5)] $V^\prime(m) > 0$.
\end{itemize}
If $V^\prime(m) > 0$, then
$$
 \sqrt{n}(\tilde{m}_n -m) \stackrel{\mathcal{D}}{\rightarrow} N(0,\tau^2)\quad \text{ with } \tau^2=\frac{l^2 \zeta}{V^\prime(m)^2}.
$$
If $V^\prime(m) = 0$, then
$$
 \sqrt{n}(\tilde{m}_n -m) \stackrel{\mathcal{D}}{\rightarrow} Y \quad \text{ in } \; \overline{\mathbb{R}},
$$
where $Y$ is uniformly distributed on $\{-\infty,\infty\}$.
\end{corollary}

\begin{proof} From (0)-(4) it follows that (A0)-(A4) are fulfilled. Furthermore, $\tilde{m}_n \in \text{Argmin}(U_n)$ a.s. by (\ref{AUn}).
Thus the assertion follows from Proposition \ref{gsmooth}.
\end{proof}

The above corollary generalizes Theorem A on p.251 in Serfling \cite{Serfling} in many respects. Firstly, there the length $l=1$. Secondly, the estimator $\tilde{m}_n$ must satisfy the equality (*) $V^+_n(\tilde{m}_n)=0$. This assumption is stronger than our (\ref{zeroineq}). Indeed, if $V^+_n(\tilde{m}_n)=0$, then by monotonicity $V^-_n(\tilde{m}_n) \le V^+_n(\tilde{m}_n) =0$, whence
(\ref{zeroineq}) is fulfilled. But, if for instance $\psi(x,t)= 1_{\{x \le t\}}-\alpha$ with $x, t \in \mathbb{R}$ and $\alpha \in (0,1)$, then
$V_n^+(t) = F_n(t)-\alpha$, where $F_n$ is the empirical distribution function pertaining to the sample $X_1,\ldots,X_n$. Consequently, an estimator
$\tilde{m}_n$ satisfying (*) cannot exist, whenever $\alpha \notin \{k/n: 1 \le k \le n-1\}$, because (*) reduces to $F_n(\tilde{m}_n)=\alpha$. In paricularly in case that $\alpha$ is an irrational number,
then for every $n \in \mathbb{N}$ no estimator exists. In contrast, the inequalities (\ref{zeroineq}) simplify to
$F_n(\tilde{m}_n-) \le \alpha \le F_n(\tilde{m}_n)$, which means that $\tilde{m}_n$ is an empirical $\alpha$-quantile.
Thirdly, it is required that $\mathbb{E}[\psi(\textbf{X},t)^2]$ is finite for $t$ in a neighborhood of $m$,
whereas we require this only for $t=m$.  As to our condition (4) notice that it follows from that third assumption by the $c_r$-inequality, see Lemma \ref{crinequality} in the Appendix below.
Fourthly, the function $\mathbb{E}[\psi(\textbf{X},t)^2]$ must be continuous at $t=m$.
Fifthly, a further assumption in Theorem A in Serfling \cite{Serfling} is that $\tilde{m}_n$ is strongly consistent, which by Lemma A on p.249 in
Serfling is guaranteed if $\psi(x,t)$ is continuous in $t$ in a neighborhood of $m$.
Finally, we have a result on distributional convergence even in the case that $V^\prime(m)=0$ (and actually if only the one-sided derivatives exist, confer Remark \ref{onesided}.)

\section{Examples in statistics}
In statistical applications the bivariate functions $h$ typically is of the form
\begin{equation} \label{hphi}
 h(\textbf{x},t)=\phi(t-k(\bf{x})),
\end{equation}
where $\phi:\mathbb{R} \rightarrow \mathbb{R}$ is convex and $k:S^l \rightarrow \mathbb{R}$ is symmetricc and measurable. W.l.o.g. we may assume that $\phi(0)=0$, because shifts along the ordinate leave minimizer unchanged. Then $\phi$ admits the
representation $\phi(t)=\int_0^t \varphi(s) ds, t \in \mathbb{R}$, where $\varphi=D^+\phi$ is increasing and rcll. Conversely, starting from
a function $\varphi$ that is increasing and rcll we can introduce $\phi(t):=\int_0^t \varphi(s) ds, t \in \mathbb{R}$. It follows that
$\phi$ is convex and that $D^+\phi(t)=\varphi(t)$ and $D^-\phi(t)=\varphi(t-)$, whence $\phi$ is differentiable at $t$, if (and only if) $\varphi$ is continuous at $t$. If $\varphi$ is strictly increasing, then $\phi$ is strictly convex. Let $R:=\mathbb{P}\circ k(\textbf{X})^{-1}$ be the distribution of $k(\textbf{X})$ and $F$ the corresponding distribution function.

Recall that $m$ minimizes $U$ if and only if
\begin{equation} \label{misminimizer}
V(m-)\le 0 \le V(m).
\end{equation}
In case that $V$ is continuous (as for instance when $\varphi$ is continuous) this is the same as $V(m)=0$.
Notice that $D^+h(\textbf{x},t)=\varphi(t-k(\textbf{x})+)=\varphi(t-k(\textbf{x}))$ and $D^-h(\textbf{x},t)=\varphi(t-k(\textbf{x})-)$.
Using the Change of variable formula we see that (A0) is satisfied if and only if
\begin{equation} \label{A0}
\int_\mathbb{R}|\varphi(t-x\pm)| R(dx) <\infty \text{ for all } t \in \mathbb{R}.
\end{equation}
In that case by (\ref{Dpm})
\begin{equation}
 V(t)= \int_{S^l} \varphi(t-k(\textbf{x})) P^l(d\textbf{x})= \int_\mathbb{R} \varphi(t-x) R(dx) \quad \forall \; t \in \mathbb{R}.
\end{equation}
Deduce from (\ref{Vnpm}) that
$$
 V_n^+(t)= \binom{n}{l}^{-1} \sum_{1 \le i_1<\ldots<i_l \le n} \varphi(t-k(X_{i_1},\ldots,X_{i_l}))
$$
and
$$
 V_n^-(t)= \binom{n}{l}^{-1} \sum_{1 \le i_1<\ldots<i_l \le n} \varphi(t-k(X_{i_1},\ldots,X_{i_l})-)
$$
By (\ref{AUn}) a point $\tilde{m}_n$ minimizes $U_n$ if and only if
\begin{equation} \label{mnsatisfiesinequalities}
 \sum_{1 \le i_1<\ldots<i_l \le n} \varphi(\tilde{m}_n-k(X_{i_1},\ldots,X_{i_l})-) \le 0 \le \sum_{1 \le i_1<\ldots<i_l \le n} \varphi(\tilde{m}_n-k(X_{i_1},\ldots,X_{i_l})).
\end{equation}
If $\varphi$ is continuous, then $V_n^+$ and $V_n^-$ coincide and therefore $\tilde{m}_n$ solves the equation
\begin{equation} \label{mnisazero}
 \sum_{1 \le i_1<\ldots<i_l \le n} \varphi(t-k(X_{i_1},\ldots,X_{i_l}))=0, \; t \in \mathbb{R}.
\end{equation}
Another use of the Change of variable formula shows that (A2) holds exactly when
\begin{equation} \label{A2}
\int_\mathbb{R} \varphi(m-x)^2 R(dx) < \infty.
\end{equation}
Further, $\zeta$ in (A3) can be rewritten as $\zeta=\int_\mathbb{R} \varphi(m-x)^2 R(dx)$, if $l=1$ and
\begin{equation} \label{A3}
\zeta=\int_S \mathbb{E}[\varphi(m-k(x,X_2,\ldots,X_l))]^2 P(dx), \text{ if } l \ge 2.
\end{equation}
Finally, (A4) reduces to
\begin{equation} \label{A4}
 \exists \; t_0<0<t_1 :  \int_\mathbb{R}\big(\varphi(m+t-x)-\varphi(m-x)\big)^2 R(dx) < \infty \;\;  \text{ for } \; t \in \{t_0,t_1\}.
\end{equation}

The existence of a minimizing point $m$ with the property (A1) is not a matter of course. Our next result gives general conditions which ensure
the existence.
Here, recall that $F$ denotes the distribution function of $R$.\\

\begin{lemma} \label{mexists} Let $\phi$ be strictly convex. Suppose that in addition
\begin{itemize}
\item[(1)]
$\phi$ is coercive with $\varphi(0-)\le 0 \le \varphi(0)$

or
\item[(2)]
$\phi$ is even and $Z:=k(X_1,\ldots,X_l)$ is symmetric (not necessarily at zero).
\end{itemize}
Morover, assume that
$\phi$ is differentiable on $\mathbb{R} \setminus \{0\}$ or equivalently that $\varphi$ is continuous on $\mathbb{R} \setminus \{0\}$.
If $F$ is continuous, then there exists a unique $m \in \mathbb{R}$ satisfying condition (A1). It fulfills the equation
\begin{equation} \label{inteq}
 \int_{\mathbb{R} \setminus \{t\}} \varphi(t-x) R(dx)=0,\; t \in \mathbb{R}.
\end{equation}
For $\phi$ differentiable on the entire real line the continuity assumption on $F$ can be dropped and the
integrals in (\ref{inteq}) are over the entire space $\mathbb{R}$.

If (2) holds, then $m$ is the center of symmetry of $Z$.
\end{lemma}

\begin{proof} The assertion follows from Lemma 2.12 in Ferger \cite{Ferger0} with $Q=R$ and $v=\phi$ there.
\end{proof}

In our first class of examples we consider very smooth functions $\varphi$.\\

\begin{example}(\textbf{smooth case 1}) \label{smooth1} Assume that $\zeta$ in (\ref{A3}) is positive. Let $\varphi$ be differentiable with a bounded derivative $\varphi^\prime$ and
strictly increasing with $\varphi(0)=0$.
If $\int x^2 R(dx)=\mathbb{E}[k(\textbf{X})^2] < \infty$ and if $R$ is not a discrete probability distribution,
then there exists a unique minimizer $m$ of $U$, such that
$$
 \sqrt{n}(\tilde{m}_n -m) \stackrel{\mathcal{D}}{\rightarrow} N(0,\tau^2)\quad \text{ with } \tau^2=\frac{l^2 \zeta}{(\int \varphi^\prime(m-x)R(dx))^2} \in (0,\infty).
$$
Here, the integral $\int \varphi^\prime(m-x)R(dx)$ in the denominator  is positive and finite.
If $\varphi^\prime>0$, then the non-discreteness assumption on $R$ can be dropped.
\end{example}

\begin{proof} Since $\varphi$ in particular is continuous, it follows that $\phi$ is differentiable. $\phi$ is strictly convex, because
$\varphi$ is strictly increasing. To see coercivity of $\phi$ consider $t>0$. Choose a fixed point $\bar{t} \in (0,t)$. Then
$$\phi(t)=\int_0^t \varphi(s) ds = \int_0^{\bar{t}} \varphi(s) ds + \int_{\bar{t}}^t \varphi(s) ds \ge \varphi(\bar{t})(t-\bar{t}).$$
Notice that by strict monotonicity $\varphi(\bar{t})>\varphi(0)=0$. Taking the limit in the above inequality results in
$\phi(t) \rightarrow \infty$ as $t \rightarrow \infty$. Analogously one shows that $\phi(t) \rightarrow \infty$ as $t \rightarrow -\infty$.
Hence $\phi$ is coercive. Thus Lemma \ref{mexists} guarantees the existence of an unique $m \in \mathbb{R}$ satisfying assumption (A1).
Next we check that the remaining A-assumptions are fulfilled. Let $c \in \mathbb{R}$ with $\varphi^\prime \le c$. To see (\ref{A0}) observe that
$$|\varphi(t-x\pm)|=|\varphi(t-x)|=|\varphi(t-x)-\varphi(0)|\le c|t-x|\le  c (|t|+|x|).$$ Consequently, $\int |\varphi(t-x\pm)|R(dx)\le c(|t|+\int |x| R(dx)) < \infty.$
Similarly, one obtains that $\int \varphi(m-x)^2 R(dx) \le 2 c^2( m^2+ \int x^2 R(dx)) < \infty$ and that $\int_\mathbb{R}(\varphi(m+t-x)-\varphi(m-x))^2 R(dx) \le 2 c^2 t^2 < \infty$, which yields (\ref{A2}) and (\ref{A4}). So all assumptions (A0)-(A4) are
fulfilled. Further, by the Monotone Convergence Theorem $V$ is continuous and it is increasing by the monotonicity of the integral. In fact,
$V$ is strictly increasing, because otherwise there are two real points $s<t$ such that $\int \varphi(t-x)-\varphi(s-x) R(dx)=0$.
Therfore $1=R(\{x \in \mathbb{R}: \varphi(t-x)=\varphi(s-x)\})=R(\emptyset)$, because $\varphi$ is strictly increasing.

Furthermore, $V$ is differentiable at $m$ with $V^\prime(m)=\int \varphi^\prime(m-x) R(dx)$ by the Differentiation lemma. Here, $V^\prime(m)>0$
because otherwiese $\int \varphi^\prime(m-x) R(dx)=0$. Since $\varphi^\prime \ge 0$ as $\varphi$ is increasing, it follows that $R(A)=1$, where
$A=\{x \in \mathbb{R}: \varphi^\prime(m-x)=0\}$. Now, $\varphi$ actually is strictly increasing, whence $B:=\{t \in \mathbb{R}: \varphi^\prime(t)=0\}$ is denumerable. But $A =\{m-t:t \in B\}$, whence $A$ is a denumerable subset of $\mathbb{R}$ with $R(A)=1$. This means that $R$ is a discrete
probability distribution in contradiction to our assumption. Thus we have verified all conditions of Proposition \ref{gsmooth}, which yields the
asymptotic normality.

If actually $\varphi^\prime >0$, then $A=\emptyset$ in contradiction to $R(A)=1$. This shows the last part of the proposition.
\end{proof}

As a simple consequence of Example \ref{smooth1} we obtain the Central Limit Theorem (U-CLT) for non-degenerate $U$-statistics.\\
%\vspace{0.3cm}
\begin{example} (\textbf{U-statistics})\label{L2} Let $\phi(t)=t^2$, so that $\varphi(t)=2 t, \varphi^\prime=2$ and $V(t)= 2(t-\mu)$, where $\mu=\int y R(dy)= \mathbb{E}[k(\textbf{X})].$ Therefore, $m=\mu$, because $V$ is continuous. Further, it follows from (\ref{mnisazero}) that $$\tilde{m}_n= \binom{n}{l}^{-1} \sum_{1 \le i_1<\ldots<i_l \le n} k(X_{i_1},\ldots,X_{i_l}),$$
i.e. $\tilde{m}_n$ is the $U$-statistic with kernel $k$.
%Since $$\delta_n(x)= \sqrt{n} V(m+a_n x)= 2 \sqrt{n} (m+a_n x -\mu)= 2 \sqrt{n} a_n x,$$
%it follows that $a_n=n^{-1/2}$ and $\delta(x)= 2x$.
Recall the kernel $K$ in (A3):
$$
 K(\textbf{x}):=D^+h(\textbf{x},m)=D^+\phi(m-k(\textbf{x}))=\varphi(m-k(\textbf{x}))=-2(k(\textbf{x})-\mu).
$$
It follows that $K_1(x)=-2(k_1(x)-\mu), x \in \mathbb{R}$, where $k_1$ is the first associate function of $k$. In (A3) it was explained that $\zeta= \text{Var}(K_1(X_1))=4 \text{Var}(k_1(X_1))$, so that $\tau^2= \frac{1}{4}l^2\zeta=l^2 \text{Var}(k_1(X_1))$. Thus Example \ref{smooth1} yields:
$$
 \sqrt{n}(\tilde{m}_n-m)= \sqrt{n}(\binom{n}{l}^{-1} \sum_{1 \le i_1<\ldots<i_l \le n} k(X_{i_1},\ldots,X_{i_l})- \mu) \rightarrow N(0,l^2 \text{Var}(k_1(X_1)).
$$
This is in accordance with the Central Limit Theorem for (non-degenerate) $U$-statistcs.

If $l=1$, then $\tilde{m}_n$ reduces to the arithmetic mean $\overline{X}_n= \frac{1}{n} \sum_{i=1}^n k(X_i)$ and if in addition $S=\mathbb{R}$ and
$k$ is equal to the identity, then $\tilde{m}_n$ further simplifies to the arithmetic mean of the data $X_1,\ldots,X_n$.
\end{example}

\vspace{0.3cm}
It shouldbe mentioned that Example \ref{smooth1} does not give a new proof of the U-CLT, because the U-CLT is used in the proof of
Theorem \ref{thm1}, which is the mother of all our examples.\\

The ''generating'' function $\varphi(t)=t$ yields the arithmetic mean (for $l=1$), while $\varphi(t)=1_{\{t \ge 0\}}-\frac{1}{2}$ gives
the median. Sigmoid functions are a compromise between these two extremes.\\

\begin{example}(\textbf{sigmoid $\varphi$)} \label{sigmoid} Let $S=\mathbb{R}, l=1$ and $k(x)=x$, whence $F$ is the distribution function of
the (real-valued) data $X_1,\ldots,X_n$.
\begin{itemize}
\item[(1)] Let $\Phi$ be the distribution function of $N(0,1)$. Consider $\varphi=\Phi-\frac{1}{2}$ and $F(x)=\Phi(x-\mu)$ with
$\mu \in \mathbb{R}$. Then $m=\mu$ by Lemma \ref{mexists} and $\tilde{m}_n$ is the unique solution of the equation
$\sum_{i=1}^n \Phi(t-X_i)=\frac{n}{2}, t \in \mathbb{R}$. An application of Example \ref{smooth1} yields
$$
 \sqrt{n}(\tilde{m}_n-\mu) \stackrel{\mathcal{D}}{\rightarrow} N(0,\frac{\pi}{3}).
$$
\item[(2)] Let $C$ be the distribution function of the standard Cauchy-distribution. Assume that $\varphi=C-\frac{1}{2}$ and $F(x)=C(x-a)$ with
$a \in \mathbb{R}$. Similarly as in (1) above, $m=a$ and $\tilde{m}_n$ is the unique solution of the equation
$\sum_{i=1}^n C(t-X_i)=\frac{n}{2}, t \in \mathbb{R}$. Another application of Example \ref{smooth1} yields
$$
 \sqrt{n}(\tilde{m}_n-a) \stackrel{\mathcal{D}}{\rightarrow} N(0,\frac{\pi^2}{3}).
$$
\end{itemize}
\end{example}

\vspace{0.2cm}
Even though the class of estimators in Example \ref{smooth1} is already really big, it does not contain the $L_p$-estiamtor
at least for $p \neq 2$. However, our next result includes also this one.\\

\vspace{0.4cm}
\begin{example} \label{smooth2} Suppose that (\ref{A0}), (\ref{A2}) and (\ref{A4}) hold and that $\zeta>0$. Let $\varphi$ be continuous and strictly increasing, concave on $[0,\infty)$, convex
on $(-\infty,0)$ with $\varphi(0)=0$ and differentiable on $\mathbb{R} \setminus \{0\}$.
Then there exists a unique minimizer $m \in  \mathbb{R}$ of $U$.
If in addition
\begin{equation} \label{phiF1}
 \frac{1}{t}(\varphi(t)-\varphi(-t))R([m,m+t]) \rightarrow 0, \; t \downarrow 0,
\end{equation}
\begin{equation} \label{phiF2}
 \frac{1}{t}(\varphi(t)-\varphi(-t))R((m-t,m]) \rightarrow 0, \; t \downarrow 0,
\end{equation}
and $R$ is not discrete, then the integral $\int_{\mathbb{R} \setminus \{m\}} \varphi^\prime(m-x)R(dx)$ is a positive and finite real number and
$$
 \sqrt{n}(\tilde{m}_n -m) \stackrel{\mathcal{D}}{\rightarrow} N(0,\tau^2)\quad \text{ with } \tau^2=\frac{l^2 \zeta}{(\int_{\mathbb{R} \setminus \{m\}} \varphi^\prime(m-x)R(dx))^2} \in (0,\infty).
$$
If actually $\varphi^\prime > 0$ on $\mathbb{R} \setminus \{0\}$, then the non-discreteness assumption can be omitted.

In case that conversely $\varphi$ is convex on $[0,\infty)$ and concave on $(-\infty,0)$, then all statements remain valid.

As to (weak) sufficient conditions for the validity of (\ref{phiF1}) and (\ref{phiF2}) we refer to Remark \ref{phiF} below.
\end{example}

\begin{proof} The existence of an unique $m \in \mathbb{R}$ with property (A1) follows  exactly as in the above proof.
And also that $V$ is continuous and strictly increasing. By Proposition \ref{gsmooth} it remains to show that $V$ is differentiable at $m$ with a positive derivative. To see this observe that for every $t>0$,
\begin{eqnarray}
& &\frac{V(m+t)-V(m)}{t}=\int_\mathbb{R} \frac{\varphi(m+t-x)-\varphi(m-x)}{t} R(dx) \nonumber\\
&=&\int_{(-\infty,m)} \frac{\varphi(m+t-x)-\varphi(m-x)}{t} R(dx)+\int_{[m,m+t]} \frac{\varphi(m+t-x)-\varphi(m-x)}{t} R(dx)\nonumber\\
&+&\int_{(m+t,\infty)} \frac{\varphi(m+t-x)-\varphi(m-x)}{t} R(dx) =: I+II+III.
\end{eqnarray}
If $x \in (-\infty,m)$, then $0<m-x<m+t-x$, whence
$$
1_{(-\infty,m)}(x)\frac{\varphi(m+t-x)-\varphi(m-x)}{t} \uparrow 1_{(-\infty,m)}(x) \varphi^\prime(m-x), \; t \downarrow 0
$$
as $\varphi$ is concave
on $[0,\infty)$, confer Theorem 23.1 of Rockafellar \cite{Rockafellar}. Therefore $$I \uparrow \int_{(-\infty,m)} \varphi^\prime(m-x) R(dx), t \downarrow 0,$$ by the Monotone Convergence Theorem upon noticing that the functions on the left side are integrable for all $t$ according to (\ref{A0}). (Later on we will see that the integral on right side is finite.)
If $x \in [m,m+t]$, then the integrand of the second integral II is greater than or equal zero and less than or equal $(\varphi(t)-\varphi(-t))/t.$ Thus
$$
0 \le II \le \frac{1}{t}(\varphi(t)-\varphi(-t)) R([m,m+t]) \rightarrow 0, \; t \downarrow 0 \; \text{ by (\ref{phiF1})}.
$$
As to the third integral III notice that if $x > m+t$, then $0>m+t-x>m-x$. Recall that the difference quotients by convexity of $\varphi$ on $(-\infty,0)$ are increasing in the variable $t>0$, confer Theorem 23.1 of Rockafellar \cite{Rockafellar}. Consequently
\begin{eqnarray*}
 0 \le 1_{(m+t,\infty)}(x) \frac{\varphi(m+t-x)-\varphi(m-x)}{t} &\le& \frac{\varphi(m+t-x)-\varphi(m-x)}{t}\\
  &\le& \frac{\varphi(m+t^*-x)-\varphi(m-x)}{t^*} \quad \forall \; t \in (0,t^*],
\end{eqnarray*}
where $t^*$ is any fixed positive real number. Thus the Dominated Convergence Theorem yields that
$$
  III \rightarrow \int_{(m,\infty)} \varphi^\prime(m-x) R(dx) \in [0,\infty), \; t \downarrow 0.
$$
This shows that $V$ is right differentiable at $m$. For the derivation of left differentiability we start
with
\begin{eqnarray}
& &\frac{V(m-t)-V(m)}{-t}=\int_\mathbb{R} \frac{\varphi(m-t-x)-\varphi(m-x)}{-t} R(dx) \nonumber\\
&=&\int_{(-\infty,m-t]} \frac{\varphi(m-t-x)-\varphi(m-x)}{-t} R(dx)+\int_{(m-t,m]} \frac{\varphi(m-t-x)-\varphi(m-x)}{-t} R(dx)\nonumber\\
&+&\int_{(m,\infty)} \frac{\varphi(m-t-x)-\varphi(m-x)}{-t} R(dx) =: A+B+C.
\end{eqnarray}
Since
$$
0 \le 1_{(-\infty,m-t]}(x)\frac{\varphi(m-t-x)-\varphi(m-x)}{-t} \le \frac{\varphi(m-t^*-x)-\varphi(m-x)}{-t^*}  \quad \forall \; t \in (0,t^*],
$$
it follows by the Dominated Convergence Theorem that $$A \rightarrow \int_{(-\infty,m)} \varphi^\prime(m-x) R(dx),\; t \downarrow 0,$$  where the integral is a finite
real number as announced above. The summand $B$ can be treated as the summand $II$ above resulting with (\ref{phiF2}) in $B \rightarrow 0$. Finally, with the help of the Monotone Convergence Theorem one shows that $C \uparrow \int_{(m,\infty)} \varphi^\prime(m-x) R(dx)$, where it is already know that the integral here is finite. Summing up we arrive at $V^\prime(m)= \int_{\mathbb{R} \setminus \{m\}} \varphi^\prime(m-x) R(dx)$. That this quantity is in fact positive follows as in the proof of Proposition \ref{smooth1}. Here, in the definitions of the sets $A$ and $B$ one simply has to replace $\mathbb{R}$ by $\mathbb{R} \setminus \{m\}$.
If convex and concave reverse their roles, the derivation of $V^\prime(m)= \int_{\mathbb{R} \setminus \{m\}} \varphi^\prime(m-x) R(dx)$ follows analogously.
\end{proof}

\begin{remark} \label{zetapositiv} If in Examples \ref{smooth1} and \ref{smooth2} above the length $l=1$, then the condition $\zeta>0$ there is automatically fulfilled provided $R$ is not equal to the Dirac-measure at point $m$. To see this recall that $\zeta=\int \varphi(m-x)^2 R(dx)$ in case $l=1$.
Assume that $\zeta=0$. Then $1=R(\{x \in \mathbb{R}: \varphi(m-x)=0\})=R(\{m\})$, because $\varphi$ is strictly increasing and $\varphi(0)=0$.
It follows that $R= \delta_m$ in contradiction to our assumption on $R$.
\end{remark}

\vspace{0.4cm}
\begin{remark} \label{phiF}
(1) Assume that $F$ has left- and right derivatives at $m$. Then in particularly, $F$ is continuous at
$m$ and thus condition (\ref{phiF1}) is fulfilled, because
$$
  \frac{1}{t}(\varphi(t)-\varphi(-t))R([m,m+t])=(\varphi(t)-\varphi(-t))\frac{F(m+t)-F(m)}{t} \rightarrow 0, t\downarrow 0.
$$
Analogously, one sees that condition(\ref{phiF2}) is satisfied as well.

(2) If $\varphi$ in Example \ref{smooth2} is differentiable at zero with $\varphi^\prime(0)< \infty$ and if $F$ is continuous at $m$, then (\ref{phiF1}) and
(\ref{phiF2}) are again fulfilled. If actually $\varphi^\prime(0)=0$, then $F$ need not be continuous at $m$.
\end{remark}

\vspace{0.3cm}
Example \ref{smooth2} enables us to handle the $L_p$-estimator.\\

%\vspace{0.3cm}
\begin{example} (\textbf{$L_p$-estimator}) For $p>1$ let $\phi(t)=|t|^p, t \in \mathbb{R}$. ($p=1$ corresponds to the median and will be examined later on.) Then $\phi$ is differentiable on $\mathbb{R}$ with continuous derivative $\varphi(t)=p \; \text{sign}(t) \; |t|^{p-1}$. Moreover, if $1<p <2$, then $\varphi$ is concave on $[0,\infty)$, convex on $(-\infty,0)$ and $\varphi(0)$. If $p \ge 2$, then concave and convex reverse their roles.
It follows with the $c_r$-inequality in Lemma \ref{crinequality} below that $$\int |\varphi(t-x)| R(dx) = p \int |t-x|^{p-1} R(dx) < \infty \; \forall \; t \in \mathbb{R},$$
if and only if $\int |x|^{p-1} R(dx) < \infty$.
Conclude from Lemma \ref{mexists} that $U$ has a unique minimizing point $m \in \mathbb{R}$ with $V(m)=0=V(m-)$, which satisfies the equation
$$
 \int_{(-\infty,m]} (m-x)^{p-1} R(dx)= \int_{(m,\infty)} (x-m)^{p-1} R(dx).
$$
Another application of the $c_r$-inequality yields that
\begin{equation} \label{73}
\int \varphi(t-x)^2 R(dx) = p^2 \int |t-x|^{2(p-1)} R(dx) \text{ is finite for all } t \in \mathbb{R}
\end{equation}
provided
$\int |x|^{2(p-1)} R(dx) < \infty$.
From (\ref{73}) it follows that (\ref{A2}) is fulfilled, but also (\ref{A4}) upon noticing that
$(\varphi(m+t-x)-\varphi(m-x))^2 \le 2(\varphi(m+t-x)^2+\varphi(m-x)^2)$.

Finally, assume that $F$ is left- and right-differentiable at $m$. Then by Remark \ref{phiF}(1) the conditions
(\ref{phiF1}) and (\ref{phiF2}) are fulfilled. In case $p>2$ these conditions are satisfied for any distribution function $F$  according to Remark \ref{phiF}(2). (The case $p=2$ is treated in Example \ref{smooth1}.)

In summary we obtain: If $\int |x|^{p-1} R(dx)$ is finite, then the minimizer $m$ of the function $t \mapsto \mathbb{E}[|t-k(\textbf{X})|^p]$  exists and is uniquely determined. Let $\tilde{m}_n$ be the $L_p$-estimator, i.e. $\tilde{m}_n$ solves the equation
$$ \sum_{\textbf{i}} 1_{\{k(\textbf{X}_{\textbf{i}})<t\}}(t-k(\textbf{X}_{\textbf{i}}))^{p-1}=\sum_{\textbf{i}} 1_{\{k(\textbf{X}_{\textbf{i}})>t\}}(k(\textbf{X}_{\textbf{i}})-t)^{p-1}, \; t \in \mathbb{R}.$$
If actually $\int |x|^{2(p-1)} R(dx) < \infty$, then
$$\sqrt{n}(\tilde{m}_n-m) \stackrel{\mathcal{D}}{\rightarrow} N(0,\tau^2) \text{ whith } \tau^2=\frac{l^2 \zeta}{p^2(p-1)^2\int_{\mathbb{R} \setminus \{m\}}|m-x|^{p-2} R(dx)}.$$
Notice that Example \ref{smooth2} in particularly guarantees the existence of the integral in the denominator.
In case $l=1$ the variance $\tau^2$ is equal to $$\frac{\int |m-x|^{2(p-1)}R(dx)}{(p-1)^2 \int_{\mathbb{R} \setminus \{m\}} |m-x|^{p-2} R(dx)}.$$
If in addition $S=\mathbb{R}$ and $k(x)=x$, then we obtain the result of Hjort and Pollard \cite{Hjort}.
\end{example}

\vspace{0.3cm}
In our next example $\varphi$ is not continuous but has a jump at point zero. Recall that $F$ is the distribution
function of $k(\textbf{X})$.\\

\begin{example} (\textbf{U-quantiles}) \label{exp2} For every fixed $\alpha \in (0,1)$ consider $\phi(t)= t (1_{\{t \ge 0\}}-\alpha).$
It follows that
$\varphi(t)= 1_{\{t \ge 0\}}-\alpha$ and therefore $$V(t)= F(t)-\alpha.$$ Thus (\ref{misminimizer}) yields that
all minimizers $m$ of $U$ are exactly those satisfing the inequalities $F(m-)\le \alpha \le F(m)$, i.e. Argmin$(U)$ is equal
to the set of all $\alpha$-quantiles of $F$. Let
$$
 F_n(x):=\binom{n}{l}^{-1} \sum_{1 \le i_1<\ldots<i_l \le n} 1_{\{k(X_{i_1},\ldots,X_{i_l})\le x\}}, \; x \in \mathbb{R}.
$$
Notice that $F_n$ is the empirical distribution function pertaining to the sample $\{k(X_{i_1},\ldots,X_{i_l}): 1 \le i_1<\ldots<i_l \le n\}$
of size $N=\binom{n}{l}$.
By (\ref{mnsatisfiesinequalities}) the induced estimator $\tilde{m}_n$ is an $\alpha$-quantile of $F_n$, i.e.
$F_n(\tilde{m}_n-)\le \alpha \le F_n(\tilde{m}_n)$. In general it is not unique. A conventional choice is
$F_n^{-1}(\alpha)$, the smallest $\alpha$-quantile. The condition (\ref{nscond}) becomes
\begin{equation} \label{deltanquantile}
 \delta_n(x)= \sqrt{n}(F(m+a_nx)-\alpha) \rightarrow  \delta(x) \quad \forall \; x \in D.
\end{equation}
It follows easily that (A0)-(A4) are fulfilled. Here, (A1) holds, if $F$ is continuous at $m$
and (A3) holds in case $l=1$, because then $\zeta=\alpha(1-\alpha)$. If $l \ge 2$, then
$\zeta=\int_S \mathbb{P}(k(x,X_2,\ldots,X_l) \le m)^2 P(dx)-\alpha^2$, whence $\zeta = 0$ leads to
$\mathbb{P}(k(x,X_2,\ldots,X_l) \le m)=\alpha$ for $P$-almost every $x \in S$, which usually results in a contradiction.
Assume that (\ref{deltanquantile}) is valid with $\delta(x) \rightarrow \pm \infty$ as $x \rightarrow \pm \infty$.
Then by Theorem \ref{thm1} we know that (\ref{dconv}) holds with $H$ being a distribution function, that
belongs to the four classes \textbf{class1}-\textbf{class4}. Infer from Theorems \ref{class1}-\ref{class3} that
$m$ is uniquely determined in case that $H$ lies in one of the first 3 classes. If $\delta(x)$ is bounded on $[0,\infty)$ or on $(-\infty,0]$,
then $H$ is a sub distribution function and we obtain distributional convergence to a $\overline{\mathbb{R}}$-valued random variable
with distribution function (for the extended real line) $H(x)=\Phi_\sigma(\delta(x)).$

In Example 4.3 of Ferger \cite{Ferger0} one finds distribution functions $F$ such that $a_n=n^{-\frac{1}{2 \beta}}$ with some $\beta>0$
and $H$ has the shape (\ref{H1}), (\ref{H2}) or (\ref{H3}). Another $F$ there yields $a_n=1/\log(n)$ and a limit $H$ as in (\ref{H4}).
As a further example consider
\begin{equation} \label{Fnewexp}
F(x):=\frac{1}{2}+\text{sign}(x) |x| (\log|x|)^2 \text{ for all } \; x \in [-\epsilon,\epsilon] \setminus \{0\}
\end{equation}
with $\epsilon > 0$ sufficiently small. Notice that $F$ is continuos, if $F(0):=\frac{1}{2}$.
Infer that $m=0$ is the unique median of $F$. It is not differentiable at $0$ and
$F^\prime(x) \rightarrow \infty$ as $t \rightarrow 0$. One easily checks with Theorem \ref{class3} that $F$ lies in \textbf{class3}. In fact, if $a_n:=n^{-\frac{1}{2}}(\log \sqrt{n})^{-2}$, then for all $x>0$ and for eventually all $n \in \mathbb{N}$ we obtain:
$$
\delta_n(x)= \sqrt{n}(F(a_n x)-\frac{1}{2})=x \big(\frac{-\log \sqrt{n}- 2 \log\log\sqrt{n}+\log(x)}{\log \sqrt{n}}\Big)^2 \rightarrow x
$$
Since $V(t)=-V(-t)$ for every $t \in [-\epsilon,\epsilon]$, it follows that $\delta_n(x) \rightarrow x$ for all $x \in \mathbb{R}$.
Consequently, $$\sqrt{n}(\log \sqrt{n})^{2} \tilde{m}_n \stackrel{\mathcal{D}}{\rightarrow} N(0,l^2 \zeta).$$
If $F$ is differentiable at $m$, then so is $V=F-\alpha$ and we obtain the square root asymptotics of Proposition \ref{gsmooth} and Remark \ref{onesided} with
$V^\prime(m)=F^\prime(m)$. In particularly,
\begin{equation}
 \sqrt{n}(\tilde{m}_n-m) \stackrel{\mathcal{D}}{\rightarrow} N(0,l^2 \zeta/F^\prime(m)^2),
\end{equation}
where $l=1$ yields the well-known limit $N(0,\alpha(1-\alpha)/F^\prime(m)^2)$.
\end{example}

\vspace{0.2cm}
The above Example \ref{exp2} in the special case $S=\mathbb{R}, l=1$ and $k(x)=x$ contains the whole theory on quantile-estimators of Smirnov's \cite{Smirnov} fundamental paper. It includes his equivalent conditions that the underlying distribution function $F$ of the data lies in the
repective domain of attraction (symbolically: $F \in \mathcal{D}(H))$. In fact, in Theorem \ref{class1}- \ref{class4} the function $V$ is given by $V(t)=F(t)-\alpha$.
In contrast to Smirnov \cite{Smirnov} we also specify the normalizing sequence $(a_n)$. For instance it follows from Theorems \ref{class1} and \ref{class3} that a possible choice is  $a_n=F^{-1}(\alpha+\frac{1}{\sqrt{n}})-F^{-1}(\alpha)$
for $F \in \mathcal{D}(H)$ with $H \in $ \textbf{class1} or \textbf{class3}.\\

Another special case in Example \ref{exp2} is $\alpha=\frac{1}{2}$ resulting in the median of $F$. Here, the corresponding
$\phi$ can be rewritten as $\phi(t)=\frac{1}{2}|t|$. The factor $\frac{1}{2}$ in front of $|t|$ is obviously unneceessary and therefore
will be ommitted in the sequel. So now we look at $h(\textbf{x},t)=|t-k(\textbf{x})|$ and additionally consider $S=\mathbb{R}.$\\

\begin{example}(\textbf{Hodges-Lehmann estimator}) Here, $k(x_1,\ldots,x_l):=l^{-1}(x_1+\ldots+x_l)$. For $l=2$ the
estimator $\tilde{m}_n$ is a median of the so-called \emph{Walsh-averages} $\{(X_i+X_j)/2: 1 \le i < j \le n\}$. This is a variant of the
Hodges-Lehmann estimator, where the Walsh-averages with $i=j$ are added to the sample.
\end{example}

\vspace{0.2cm}
If in the above example the function $k$ more generally is given by $k(x_1,x_2)= \beta x_1+(1-\beta) x_2$ with $\beta>0$ a fixed constant,
then the resulting median is a variant of the estimator due to \textbf{Maritz, Wu and Staudte} \cite{Maritz}.\\

\begin{example}(\textbf{Bickel-Lehmann estimator}) This estimator is defined as a median of the sample $\{|X_i-X_j|: 1 \le i < j \le n\}$
and corresponds to $k(x_1,x_2)=|x_1-x_2|.$
\end{example}

\vspace{0.2cm}
In our next example $S=\mathbb{R}^2$.\\

\begin{example}(\textbf{Theil-Sen estimator}) Let $X_i=(Y_i,Z_i), 1 \le i \le n,$ and $\{k(X_i,X_j):=\frac{Z_i-Z_j}{Y_i-Y_j}: 1 \le i<j \le n\}$, which
is a sample of well-defined real random variables, if for instance $Y_1$ is continuously distributed. Then the median of the sample is known to be a robust estimator for the slope $\beta$ in the simple linear regression model $Z_i=\alpha+\beta Y_i$.
\end{example}

\vspace{0.2cm}
The function $\varphi(t)= 1_{\{t \ge 0\}}-\alpha$, which induces the $\alpha$-quantile has a jump at point zero (and is continuous elsewhere).
Let us more generally consider monotone increasing functions $\varphi$, continuous on $\mathbb{R} \setminus \{0\}$ and with a jump at zero. Assume that $\kappa+:=\varphi(0) \ge 0$ and $\kappa_-:=\varphi(0-)\le 0$ and that $\kappa_*:=\min\{\kappa_+,-\kappa_-\} > 0$.
Then $\kappa:=\kappa_+-\kappa_- > 0$ is the jump height. Notice that $\varphi$ still satisfies the
assumption of Lemma \ref{mexists}, whence there exists a unique $m \in \mathbb{R}$ with $V(m)=0=V(m-)$ so that (A1) is fulfilled.
Moreover, $\int \varphi(m-x)^2 R(dx) \ge \kappa_*^2 >0$, so that (A3) is fulfilled at least in the important special case $l=1$.
Introduce
\begin{equation} \label{phic}
 \varphi_c := \varphi -\kappa_+ 1_{[0,\infty)}-\kappa_- 1_{(-\infty,0)}.
\end{equation}
Observe that $\varphi_c(0+)=\varphi(0)-\kappa_+=0$ and $\varphi_c(0-)=\varphi(0-)-\kappa_-=0$. Thus $\varphi_c$ is continuous at zero and $\varphi_c(0)=0$.
Geometrically, the graph of $\varphi_c$ is created from $\varphi$ by moving the two branches to the left and right of the ordinate to the origin and connecting them there. Notice that $\varphi$ meets the requirements (\ref{A0}) and (\ref{A2}) if and only if this is true for $\varphi_c$. Indeed, it is convenient to formulate our conditions in terms of $\varphi_c$ rather than for $\varphi$.
Moreover from (\ref{phic}) we obtain a decomposition of $V$ as follows:
\begin{equation} \label{Vc}
 V(t)= \kappa F(t)+ V_c(t) + \kappa_-,
\end{equation}
where $V_c(t)= \int \varphi_c(t-x) R(dx)$.
Thus $\delta_n$ has the following shape:
\begin{equation} \label{deltanj}
 \delta_n(x)= \kappa \sqrt{n}(F(m+a_n x)-F(m))+ \sqrt{n} (V_c(m+a_n x)-V_c(m)).
\end{equation}
Assume that
\begin{equation} \label{DomainF}
 \bar{\delta}_n(x) := \sqrt{n}(F(m+a_n x)-F(m)) \rightarrow \bar{\delta}(x) \quad \forall \; x \in D,
\end{equation}
where $D$ is a dense subset of $\mathbb{R}$ and $\bar{\delta}$ belongs to one of the four classes in (\ref{d1})-(\ref{d4}). In other words, $F \in \mathcal{D}(H)$. If $\varphi_c$ satifies the smoothness-conditions of the $\varphi$'s in Examples \ref{smooth1} or \ref{smooth2}, then our proofs show that $V_c$ is differentiable at $m$ with $V_c^\prime(m)=\int \varphi_c^\prime(m-x) R(dx) \ge 0$. Suppose that
$$
 \sqrt{n} a_n \rightarrow \rho.
$$
Then by (\ref{deltanj})
\begin{eqnarray}
 \delta_n(x) &=& \kappa \sqrt{n}(F(m+a_n x)-F(m))+ \sqrt{n} a_n \frac{V_c(m+a_n x)-V_c(m)}{a_n x} x \nonumber\\
             &\rightarrow& \kappa \bar{\delta}(x)+\rho V_c^\prime(m) x = \delta(x) \quad \text{ for all } x \in D. \label{deltarhopositiv}
\end{eqnarray}
Notice that $\rho \ge 0$.\\

\begin{example} \label{rhogleich0}
If $\rho=0$, then all statements of Theorem \ref{thm1} hold with $\delta=\kappa \bar{\delta}$.
\end{example}

\vspace{0.2cm}
If for instance $F$ is as in (\ref{Fnewexp}), then $\sqrt{n} a_n= (\log \sqrt{n})^{-2} \rightarrow 0$. Or if $a_n=n^{-\frac{1}{2 \beta}}$,
then $\rho=0$, whenever $\beta < 1$.

\vspace{0.2cm}
\begin{example} \label{rhopositive}
Consider $0 < \rho < \infty$.
Observe that for every $x>0$ we have:
$$
 \frac{F(m+a_n x)-F(m)}{a_n x}= \frac{\bar{\delta}_n(x)}{\kappa \sqrt{n} a_n x} \quad \forall \; n \in \mathbb{N}.
$$
Assume that $\bar{\delta}$ has the shape (\ref{d1}).
Taking the limit $n \rightarrow \infty$ shows that $F$ has a positive right derivative $D^+F(m)= \bar{\delta}(x)/(\kappa \rho x) >0$, because $\bar{\delta}(x)>0$. Rearranging the last equality leads to $\bar{\delta}(x)= \kappa \rho  D^+F(m) x$ for all $x>0$. Consequently, $\delta$ in (\ref{deltarhopositiv}) is given by
$\delta(x)= \rho (\kappa D^+F(m)+ V_c^\prime(m)) x$, if $x>0$ and equal to $-\infty$, if $x<0$. Thus with Theorem \ref{thm1} and Slutsky's Theorem
we arrive at
\begin{equation} \label{dconvdiscontphi}
 \sqrt{n}(\tilde{m}_n-m) \stackrel{\mathcal{D}}{\rightarrow} Y,
\end{equation}
where
$$
 \mathbb{P}(Y \le x) = \left\{ \begin{array}{l@{\quad,\quad}l}
                 0 & x<0\\ \Phi_\sigma\Big((\kappa D^+F(m)+ V_c^\prime(m)) x)\Big) & x > 0.
               \end{array} \right.
$$
Analogously, we can treat the two cases $\bar{\delta}$ of shape (\ref{d2}) or (\ref{d3}) and obtain the $\sqrt{n}$-distributional convergence (\ref{dconvdiscontphi}). Here in case (\ref{d2}),
$$
 \mathbb{P}(Y \le x) = \left\{ \begin{array}{l@{\quad,\quad}l}
                \Phi_\sigma\Big((\kappa D^-F(m)+ V_c^\prime(m)) x)\Big) & x<0\\ 1 & x > 0.
               \end{array} \right.
$$
and otherwise
$$
 \mathbb{P}(Y \le x) = \left\{ \begin{array}{l@{\quad,\quad}l}
                \Phi_\sigma\Big((\kappa D^-F(m)+ V_c^\prime(m)) x)\Big) & x<0\\ \Phi_\sigma\Big((\kappa D^+F(m)+ V_c^\prime(m)) x)\Big) & x > 0.
               \end{array} \right.
$$
Thus, if $F$ is differentiable at $m$, then
$$
 \sqrt{n}(\tilde{m}_n-m) \stackrel{\mathcal{D}}{\rightarrow} N(0,\tau^2) \quad \text{ with } \tau^2= \frac{l^2 \zeta}{(\kappa D^-F(m)+ V_c^\prime(m))^2}.
$$
Formally, for $\kappa=0$ we obtain the asymptotic normality in Examples \ref{smooth1} and \ref{smooth2}, because then $\varphi_c = \varphi$ and hence
$V_c^\prime(m)$ is equal to $\int \varphi^\prime(m-x)R(dx)$ or $\int_{\mathbb{R} \setminus \{m\}} \varphi^\prime(m-x)R(dx)$ according as $\varphi$ is differentiable at $0$ or not.

Finally, assume that $\bar{\delta}$ is of type (\ref{d4}). Then by (\ref{deltarhopositiv})
$$
\delta_n(x) \rightarrow \delta(x) = \left\{ \begin{array}{l@{\quad,\quad}l}
                 -\infty & x<-c_1\\ \rho V_c^\prime(m) x   & -c_1 < x < c_2\\ \infty & x > c_2.
               \end{array} \right. \quad (c_1,c_2 \ge 0, \max\{c_1,c_2\}>0)
$$
and by Theorem \ref{thm1} the limit $H(x)=\Phi_\sigma(\delta(x))$ is a distribution function not degenerated at $0$.
By the second part of Theorem \ref{thm1} $H$ can only belong to the four classes \textbf{class1}-\textbf{class4}.
This results in a contradiction when $V_c^\prime(m)$ is positive. However, if (and only if) $V_c^\prime(m)=0$, then we obtain
that
$$
 \frac{\tilde{m}_n-m}{a_n} \stackrel{\mathcal{D}}{\rightarrow} Y,
$$
where $Y$ is uniformly distributed on $\{-c_1,c_2\}$.
\end{example}

\vspace{0.2cm}
\begin{example} \label{rhoisinfinity} Assume that $\rho=\infty$. If $V_c^\prime(m)>0$, then it follows from (\ref{deltarhopositiv}) that
$\delta(x)= \infty$ and $\delta(x)=-\infty$ according as $x>0$ or $x<0$. Thus Theorem \ref{thm1} yields that
$(\tilde{m}_n-m)/a_n \stackrel{\mathbb{P}}{\rightarrow} 0$. However, if $V_c^\prime(m)=0$, then all statements of Theorem \ref{thm1} hold with $\delta=\kappa \bar{\delta}$.
\end{example}

\vspace{0.2cm}
\begin{example}
Finally, we would like to consider convex functions $\phi$, which are picewise linear. The corresponding ''generating'' function $\varphi$
is a step-function. For the sake of simplicity let us consider a three-step function $\varphi = \alpha 1_{(-\infty,0)}+ \beta 1_{[0,r)}+ \gamma 1_{[r,\infty)}$ with $r>0$ and $\alpha<0<\beta<\gamma$. We obtain $V(t)=\alpha+(\beta-\alpha) F(t)+(\gamma-\beta) F(t-r)$. By Lemma \ref{mexists}
there exists a unique $m \in \mathbb{R}$ such that $V(m)=0$. Therefore,
$$\delta_n(x)=\sqrt{n} V(m+a_n x)=(\beta-\alpha) \delta_{n,0}(x)+(\gamma-\beta) \delta_{n,r}(x),$$ where $\delta_{n,0}(x)= \sqrt{n}(F(m+a_n x)-F(m))$ and $\delta_{n,r}(x)= \sqrt{n}(F(m-r+a_n x)-F(m-r))$. If $\delta_{n,0}(x) \rightarrow \delta_0(x)$ and $\delta_{n,r}(x) \rightarrow \delta_r(x)$ for all $x \in D$, then $$\delta_n(x) \rightarrow \delta(x)= (\beta-\alpha) \delta_0(x)+(\gamma-\beta) \delta_r(x) \text{ for all } x \in D.$$

Suppose for instance that $\delta_0(x)$ or $\delta_r(x)$ converge to $\pm \infty$ as $x \rightarrow \pm \infty$, then $\delta(x) \rightarrow \pm \infty$ as $x \rightarrow \pm \infty$, so that $H(x)=\Phi_\sigma(\delta(x))$ is a distribution function. If $H$ is not degenerated at $0$, we obtain from Theorem \ref{thm1} distributional convergence with $H$ in one of the four classes.
\end{example}

\section{Appendix}
In this section, we derive results that may seem purely technical, but are extremely useful for our purposes.\\

\begin{lemma} \label{LA1} If (A1) and (A2) hold, then
\begin{equation} \label{2mom}
\mathbb{E}[D^+h(X_1,\ldots,X_l,m)^2]=\mathbb{E}[D^-h(X_1,\ldots,X_l,m)^2].
\end{equation}
If (A4) holds, then
\begin{equation} \label{incr}
 \mathbb{E}[\big(D^+h(X_1,\ldots,X_l,m+t)-D^+h(X_1,\ldots,X_l,m)\big)^2] < \infty \quad \forall \; t \in [t_0,t_1].
\end{equation}
\end{lemma}

\begin{proof} Let $\textbf{X}:=(X_1,\ldots,X_l)$. First notice that $D^+h(\textbf{X},m)$ and $D^-h(\textbf{X},m)$ both are integrable by (A2). Deduce from (A1) and (\ref{Dpm}) that $\mathbb{E}[D^+h(\textbf{X},m)-D^-h(\textbf{X},m)]=0$, whence \begin{equation} \label{Pas}
D^+h(\textbf{X},m)=D^-h(\textbf{X},m)\; \mathbb{P}-\text{a.s},
\end{equation}
because $D^+h(\textbf{x},\cdot) \ge D^-h(\textbf{x},\cdot)$ for all $\textbf{x} \in S^l$ by Theorem 24.1 in Rockafellar \cite{Rockafellar}. Consequently, $\mathbb{E}[D^+h(\textbf{X},m)^2]=\mathbb{E}[D^-h(\textbf{X},m)^2]$.

Let $t \in [0,t_1]$. By Theorem 24.1 in Rockafellar \cite{Rockafellar} $D^+h(\textbf{x},\cdot)$ is increasing for each $\textbf{x} \in S^l$ . It follows that
$$0 \le D^+h(\textbf{X},m+t)-D^+h(\textbf{X},m) \le D^+h(\textbf{X},m+t_1)-D^+h(\textbf{X},m)$$ and thus
$0 \le \big(D^+h(\textbf{X},m+t)-D^+h(\textbf{X},m)\big)^2 \le \big(D^+h(\textbf{X},m+t_1)-D^+h(\textbf{X},m)\big)^2$.
Conclude from (A4) that
$$
 \mathbb{E}[\big(D^+h(\textbf{X},m+t)-D^+h(\textbf{X},m)\big)^2] < \infty \quad \forall t \in [0,t_1]
$$
In the same fashion one treats the case $t \in [t_0,0]$.
\end{proof}

\begin{lemma} \label{LA2} If (A0)-(A4) hold, then so do (B0)-(B4). Moreover, $\zeta=\zeta^-$.
\end{lemma}

\begin{proof} Assume that (A1)-(A4) are true. Obviously, (A0) and (A1) imply (B0) and (B1), respectively. Infer from (\ref{2mom}) of Lemma \ref{LA1} that (B2) is satisfied.

To see the validity of (B3) we first simplify our notation: $f^\pm(\textbf{x}):=D^\pm h(\textbf{x},m)$. A reformulation of (\ref{Pas}) yields: $P^l(N)=0$, where $N=\{\textbf{x} \in S^l: f^+(\textbf{x}) \neq f^-(\textbf{x})\}$.
An application of Fubini's theorem yields
that $0= \int_S P^{l-1}(N_{x_1}) P(dx_1)$ with $N_{x_1}=\{(x_2,\ldots,x_l) \in S^{l-1}: (x_1,x_2,\ldots,x_l) \in N\}$. As the integrand is non- negative, so that $P^{l-1}(N_{x_1})=0 \; P$-almost surely, i.e. $P(E)=1$, where $E=\{x_1 \in S: P^{l-1}(N_{x_1})=0\}=\{x_1 \in S: P^{l-1}(N_{x_1}^c)=1\}$. Here,
\begin{eqnarray*}
N_{x_1}^c&=&\{(x_2,\ldots,x_l) \in S^{l-1}: (x_1,x_2,\ldots,x_l) \notin N\}\\
&=&\{(x_2,\ldots,x_l) \in S^{l-1}: f^+(x_1,x_2,\ldots,x_l)=f^-(x_1,x_2,\ldots,x_l)\}.
\end{eqnarray*}
With these preparations we can conclude as follows:
\begin{eqnarray*}
\zeta^- &=& \int_S [\int_{S^{l-1}} f^-(x_1,\ldots,x_l) d P^{l-1}(x_2,\ldots,x_l)]^2 d P(x_1)\\
      &=& \int_E [\int_{N_{x_1}^c} f^-(x_1,\ldots,x_l) d P^{l-1}(x_2,\ldots,x_l)]^2 d P(x_1)\\
      &=& \int_E [\int_{N_{x_1}^c} f^+(x_1,\ldots,x_l) d P^{l-1}(x_2,\ldots,x_l)]^2 d P(x_1)\\
      &=& \int_S [\int_{S^{l-1}} f^+(x_1,\ldots,x_l) d P^{l-1}(x_2,\ldots,x_l)]^2 d P(x_1)\\
      &=& \zeta.
\end{eqnarray*}
Especially (B3) holds true.

As to (B4) recall that $\mathbb{E}[D^\pm h(\textbf{X},t)]=D^\pm U(t)$ for every $t \in \mathbb{R}$ by (\ref{Dpm}) and that
the function $U$ is convex. Thus by Theorem 24.1 of Rockafellar \cite{Rockafellar} $D^+U$ and $D^-U$ both are increasing. Moreover, $D^-U \le D^+U$ and the set $J :=\{t \in \mathbb{R}: D^-U(t)<D^+U(t)\}$ is denumerable, confer Niculescu and Persson \cite{Niculescu}, p. 20. Consequently, the complement $\Delta:=J^c=\{D^+U=D^-U\}$ is dense in $\mathbb{R}$ and $\mathbb{E}[D^+h(\textbf{X},t)-D^-h(\textbf{X},t)]=0$ for all $t \in \Delta$ , where the integrand is non-negative. Therefore
\begin{equation}
D^+h(\textbf{X},t)=D^-h(\textbf{X},t)\; \mathbb{P}\text{-a.s. for all } t \in \Delta.
\end{equation}
By denseness we find points $u_0 \in [m+t_0,m)$ and $u_1 \in (m,m+t_1]$ such that $D^+h(\textbf{X},u_0)=D^-h(\textbf{X},u_0)$ and
$D^+h(\textbf{X},u_1)=D^-h(\textbf{X},u_1)$ a.s. Put $s_0:=u_0-m$ and $s_1=u_1-m$. Taking into account (\ref{Pas}) we now have that
$$
D^+h(\textbf{X},m+s_0)=D^-h(\textbf{X},m+s_0),\; D^+h(\textbf{X},m+s_1)=D^-h(\textbf{X},m+s_1)
$$
and $D^+h(\textbf{X},m)=D^-h(\textbf{X},m)$ with probability one. Together with (\ref{incr}) this ensures (B4), because $t_0\le s_0<0<s_1\le t_1$.
\end{proof}

Recall that $\delta_n(x)=\sqrt{n} D^+U(m+a_n x)$ and $\delta_n^*(x)=\sqrt{n} D^-U(m+a_n x)$.\\

\begin{lemma} \label{LA3}
\hspace{0.1cm}
\begin{itemize}
\item[(1)] There exists a subset $D_1=A^c$, which is the complement of a countable set $A$ such that $\delta_n^*(x)=\delta_n(x)$ for all $x \in D_1$ and for all $n \in \mathbb{N}$.
\item[(2)]
Assume that
\begin{equation*}
 \delta_n(x) \rightarrow \delta(x)\in [-\infty,\infty] \quad \forall\; x \in D,
\end{equation*}
where $D$ lies dense in $\mathbb{R}$. Then there exists a dense subset $D_2$ with $D_2 \subseteq D_1$ such that
\begin{equation*}
 \delta_n^*(x) \rightarrow \delta(x)\in [-\infty,\infty] \quad \forall\; x \in D_2.
\end{equation*}
\end{itemize}
\end{lemma}

\begin{proof} (1) Recall form the above proof that $J=\{t \in \mathbb{R}: D^-U(t) \neq D^+U(t)\}$ is denumerable.
Notice that $\delta_n(x) \neq \delta_n^*(x)$ if and only if $x \in \{(y-m)/a_n: y \in J\} =:A_n$. Here, $A_n$ is denumerable and
so is $A:=\cup_{n \in \mathbb{N}} A_n$. Thus $D_1:= A^c = \cap_{n \in \mathbb{N}} A_n^c$ is a dense subset of $\mathbb{R}$.
Moreover, if $x \in D_1$, then $\delta_n(x)=\delta_n^*(x)$ for every $n \in \mathbb{N}$.

(2) Conclude from the assumption that $F_n(x):=\Phi_1(\delta_n(x))\rightarrow \Phi_1(\delta(x))=:F(x)$ for all $x \in D$. All involved functions are increasing. Therefore by Lemma 5.74 in Witting and M\"{u}ller-Funk \cite{Witting} one has that $F_n(x) \rightarrow F(x)$ for all $x \in C_F$. Since $\Phi_1$ is invertible, we now know that $\delta_n(x) \rightarrow \delta(x)$ for all $x \in C_F$. By monotonicity of $F$ the complement $D_F$ of $C_F$ is countable and so is $A \cup D_F$, whence $D_2:=(A \cup D_F)^c=D_1 \cap C_F$ is dense.
Conclude with part (1) that $\delta_n^*(x) =\delta_n(x) \rightarrow \delta(x)$ for each $x \in D_2$.
\end{proof}

The following inequality is a well-known result from analysis, which we state for the sake of convenience.\\

\begin{lemma}(\textbf{$c_r$-inequality}) \label{crinequality} $|u+v|^r \le c_r(|u|^r+|v|^r)$ for all $u,v \in \mathbb{R}$ and for all $r>0$, where
$c_r=1$ or $c_r=2^{r-1}$ according as $r \le 1$ or $r>1$.

\end{lemma}

\vspace{1cm}
\textbf{Declarations}\\

\textbf{Compliance with Ethical Standards}: I have read and I understand the provided information.\\

\textbf{Competing Interests}: The author has no competing interests to declare that are relevant to the content of
this article.


\end{document}